\documentclass[a4paper]{article}

\usepackage{amsmath,amssymb,amsbsy}
  \usepackage{paralist}
  \usepackage{graphics} 
  \usepackage{epsfig} 
\usepackage{amsthm}
\usepackage{tikz}
\usepackage[all]{xy}
 \usepackage{verbatim}
\usepackage[bookmarks]{hyperref}

\newtheorem{theorem}{Theorem}[section]
\newtheorem{proposition}[theorem]{Proposition}
\newtheorem{corollary}[theorem]{Corollary}
\newtheorem{lemma}[theorem]{Lemma}
\newtheorem{main lemma}[theorem]{Main Lemma}

{\theoremstyle{definition}

\newtheorem{remark}{Remark}

\newcommand{\R}{\mathbb{R}}

\newcommand{\N}{\mathbb{N}}

\newcommand{\mf}[1]{\mathbf{#1}}

\def\pa{\partial}

\def\wc{\rightharpoonup}

\begin{document}






\begin{center}
\Large{\sc{On existence and phase separation of solitary waves for nonlinear Schr\"odinger systems modelling simultaneous cooperation and competition}}
\end{center}

{\centerline{Nicola Soave}
\medskip
\footnotesize
 \centerline{Justus-Liebig-Universit\"at Giessen,}
 \centerline{Mathematisches Institut,}
   \centerline{Arndtstr. 2, 35392 Giessen, Germany}
\centerline{email: nicola.soave@gmail.com; nicola.soave@math.uni-giessen.de}   }

\begin{abstract}
\noindent We study existence and phase separation, and the relation between these two aspects, of positive bound states for the nonlinear elliptic system
\[
\begin{cases}
- \Delta u_i + \lambda_i u_i = \sum_{j=1}^d \beta_{ij} u_j^2 u_i & \text{in $\Omega$} \\
u_1 =\cdots = u_d=0 & \text{on $\pa \Omega$}.
\end{cases}
\]
This system arises when searching for solitary waves for the Gross-Pitaevskii equations.
We focus on the case of \emph{simultaneous cooperation and competition}, that is, we assume that there exist two pairs $(i_1,j_1)$ and $(i_2,j_2)$ such that $i_1 \neq j_1$, $i_2 \neq j_2$, $\beta_{i_1 j_1}>0$ and $\beta_{i_2,j_2}<0$. Our first main results establishes the existence of solutions with at least $m$ positive components for every $m \le d$; any such solution is a minimizer of the energy functional $J$ restricted on a \emph{Nehari-type manifold} $\mathcal{N}$. At a later stage, by means of level estimates on the constrained second differential of $J$ on $\mathcal{N}$, we show that, under some additional assumptions, any minimizer of $J$ on $\mathcal{N}$ has all nontrivial components. In order to prove this second result, we analyse the phase separation phenomena which involve solutions of the system in a \emph{not completely competitive framework}.
\end{abstract}


\section{Introduction and main results}

Starting from the pioneering paper \cite{LinWeiRN}, the nonlinear system of elliptic equations
\begin{equation}\label{systemdcomp}
\begin{cases}
- \Delta u_i + \lambda_i u_i = \sum_{j=1}^d \beta_{ij} u_j^2 u_i & \text{in $\Omega$} \\
u_1 =\cdots = u_d=0 & \text{on $\pa \Omega$},
\end{cases} \quad i=1,\dots,d,
\end{equation}
defined on a possibly unbounded domain $\Omega \subset \R^N$ with $N =2,3$, has been intensively studied, due to its wide applicability in several physical contexts: it appears, e.g, in the Hartree-Fock theory for Bose-Einstein condensates with multiple states (see \cite{AnEnMa,Gr61,MyBuGh,PaWa,PiSt,RuCaFu}). The so-called Bose-Einstein condensation occurs when, in a dilute gas of bosons, most of the particles occupy the same quantum state, so that they can be described by the same wave-function. This phenomenon, which has been theorized by Einstein starting from some observations by Bose, has been experimentally observed only in recent years. If the gas is a mixture of $d$ components, condensation occurs between particles of the same species, and to describe the evolution of the resulting wave-functions $\psi_1,\dots,\psi_d$ it has been proposed the following system of Schr\"odinger equations:
\[
\begin{cases}
-i \frac{\pa}{\pa t}\psi_j= \Delta \psi_j + \sum_{k = 1}^d \beta_{ij} |\psi_k|^2 \psi_j & \text{in $\R^+ \times \Omega$} \\
\psi_j \in H_0^1(\Omega;\mathbb{C}) \text{ for every $t>0$}
\end{cases} \quad j=1,\dots,d,
\]
known in the literature as the Gross-Pitaevskii equations. As usual in quantum mechanics, $|\psi_j(t,x)|^2$ represents the probability of finding a particle of the $j$-th species in the point $x$ at time $t$; in the model $\beta_{ii}$, $\beta_{ij}=\beta_{ji}$ are the intraspecies and the interspecies scattering length, respectively. The sign of $\beta_{ii}$ describes the interaction between particles of the same condensate: $\beta_{ii}>0$ means attractive (or cooperative) interaction, while $\beta_{ii}<0$ means repulsive (or competitive) interaction. Analogously, $\beta_{ij}$ describes the interaction between particles of two different condensates. We always consider the so-called \emph{self-focusing case}, characterized by $\beta_{ii}>0$ for every $i$. When searching for solutions as \emph{solitary waves}, that is, in the form $\psi_j(t,x)=e^{-i \lambda_j t} u_j(x)$ for some $\lambda_j>0$, one is led to study system \eqref{systemdcomp}. 

Problem \eqref{systemdcomp} arises also in nonlinear optics, in the study of beams in Kerr-like photorefractive media (see \cite{AkAn,BSSSC,Hi}), and as a variational model in population dynamics (this interpretation has been proposed for similar systems in \cite{CoTeVe2002,CoTeVe2003}, see also \cite{CaCo} for an overview of the role of reaction-diffusion equations in this field). 

This paper concerns the existence of solutions of \eqref{systemdcomp} \emph{having all nontrivial components}, from now on simply called \emph{nontrivial} solutions, in cases of simultaneous cooperation and competition; that is, we assume that there exist two pairs $(i_1,j_1)$ and $(i_2,j_2)$ such that $i_1 \neq j_1$, $i_2 \neq j_2$, $\beta_{i_1 j_1}>0$ and $\beta_{i_2,j_2}<0$. Under this assumption, we provide sufficient conditions on the data of the problem in order to find solutions \emph{having all strictly positive components} in $\Omega$; from now on, we refer to such a solution as to a \emph{positive solution}. In order to motivate our study, in what follows we give a brief review of the known existence results regarding system \eqref{systemdcomp}.

In light of the symmetry $\beta_{ij}=\beta_{ji}$, problem \eqref{systemdcomp} has variational structure, as its solutions are critical points of the energy functional
\[
J(u_1,\dots, u_d)= \frac12 \int_{\Omega}  \sum_{i=1}^d \left( |\nabla u_i|^2 +\lambda_i u_i^2 \right) - \frac{1}{4} \int_{\Omega} \sum_{i,j=1}^d \beta_{ij} u_i^2 u_j^2.
\]
Note that $J$ is defined and differentiable for $(u_1,\dots,u_d) \in (H^1_0(\Omega))^d$, due to the continuous embedding $H_0^1(\Omega) \hookrightarrow L^4(\Omega)$, which holds for $N \le 4$ and is subcritical for $N \le 3$. As observed in \cite{Si}, the functional $J$ has a certain scalar structure in the product space $(H^1_0(\Omega))^d$, and this allows, under weak assumptions on the coupling parameters, to apply the classical tools of the critical point theory, such as the mountain pass lemma or constrained minimization on the Nehari manifold; this permits to obtain a solution $(u_1,\dots, u_d) \neq (0,\dots,0)$, but does not exclude, a priori, that this solution is \emph{semi-trivial}, that is, some of its components $u_i$ are identically zero. Being interested in solutions having all non-zero components, this is surely one of the main difficulties to face, as one can easily realize by thinking at the known results for the $2$ components system.
It has been shown in \cite{AmColor, BaWaWe, Si}, independently and with different methods, that there exist real numbers $0<\Lambda_1 \le \Lambda_2$ (in general the strict inequality holds) depending on $\lambda_1,\lambda_2, \beta_{11}, \beta_{22}>0$, such that if either $\beta_{12} < \Lambda_1$, or $\beta_{12} > \Lambda_2$, then there exists a positive solution $(\bar u, \bar v)$ which is a \emph{bound state}, that is, it has finite energy. On the contrary, if $\Lambda_1 \le \beta_{12} \le \Lambda_2$, then a solution with all nonnegative and nontrivial components does not always exists. This fact follows as a particular case of a more general result (stated in \cite{BaWa,Si}) concerning systems with an arbitrary number of components, and which can be generalized also in a non-autonomous setting, see the forthcoming Proposition \ref{prop: non-ex}. 

It can be useful to comment some of the methods which have been employed to obtain the existence of a positive bound state in the different situations $\beta_{12}< \Lambda_1$, or $\beta_{12}> \Lambda_2$. In \cite{Si}, assuming that $\beta_{12}$ is sufficiently small, a positive solution has been found by minimizing the functional $J$ on the constraint
\[
\mathcal{M}:= \left\{ (u,v) \in \left(H^1_r(\R^N)\right)^2 \left| \begin{array}{l}
u,v \not \equiv 0, \ \pa_u J(u,v)u = 0, \\ \pa_v J(u,v)v=0
\end{array} \right.  \right\}, 
\]
where $\pa_u J$ and $\pa_v J$ denote the partial derivatives of $J$ with respect to $u$ and $v$, respectively, and $H^1_r(\R^N)$ denotes the space of radially symmetric $H^1(\R^N)$ functions; note that any function in $\mathcal{M}$ has both nontrivial components, so that it is sufficient to show that there exists a minimizer of $J$ on $\mathcal{M}$, and that any such minimizer is a solution of the considered problem, in order to have a nontrivial solution; this approach, which was already adopted in similar situations in \cite{CoTeVe2002, CoTeVe2003, LiWe2}, fails when the coupling parameter becomes sufficiently large. In this last case, the existence of a positive bound state has been proved in \cite{AmColor,Si} (and also in \cite{MaMoPe} with a slightly different method) by firstly minimizing the functional $J$ on the classical Nehari manifold
\[
\mathcal{M}^*:= \left\{ (u,v) \in \left(H^1(\R^N)\right)^2: (u,v) \not \equiv (0,0), \ \langle \nabla J(u,v),(u,v)\rangle=0 \right\},
\]
and then by showing that, provided $\beta_{12}$ is sufficiently large, it results
\begin{equation}\label{eq101}
\inf_{(u,v) \in \mathcal{M}^*} J(u,v) < \inf\left\{J(u,v): (u,v) \in \mathcal{M}^* \text{ and either $u \equiv 0$ or $v \equiv 0$}\right\}.
\end{equation}
We point out that the existence of a minimizer on $\mathcal{M}^*$ is not sufficient to find a nontrivial solution, and in all the quoted works the authors overcame the problem with careful level estimates which allow them to prove the \eqref{eq101}.

These approaches have been extended for systems with $d > 2$, in both \emph{purely cooperative cases}, i.e. $\beta_{ij}>0$ for every $i,j$, and \emph{purely competitive cases}, i.e. $\beta_{ij} \le 0$ for every $i\neq j$; in the former case, we refer to Theorem 2 in \cite{LinWeiRNerr} (see also \cite{LinWeiRN}), the results of  Section 6 in \cite{AmColor} (see also \cite{Color}), Theorem 2.1 and Corollary 2.3 in \cite{LiWa10}, and the results of Section 4 in \cite{Si}; in the latter one, we remind the reader to Theorem 3.1 of \cite{LiWa10}.

As already anticipated, our paper concerns the complementary situation in which cooperation and competition coexist. As far as we know, the only two available results in this setting are Theorem 4 in \cite{LinWeiRN}, where a very specific situation is considered, and Theorem 2.1 in \cite{LiWa08}; therein, the authors considered the general $d$ components system \eqref{systemdcomp} in $\R^N$ and obtained existence and multiplicity of nontrivial (neither necessarily minimal, nor positive) solutions assuming that $|\beta_{ij}|$ is sufficiently small for every $i \neq j$.

We can extensively enlarge the set of the coupling parameters for which a positive solution does exist; furthermore, in several situations we obtain \emph{least energy positive solutions}. A least energy positive solution $\mf{u}$ of \eqref{sysx} is a solution with minimal energy among all the positive solutions. To state and discuss our main results, which regard more general non-autonomous problems, we introduce some notation and terminology.

Before, we think that it is convenient to complete the bibliographic introduction mentioning other results concerning both existence, multiplicity, and qualitative properties of bound and ground states \cite{BaDaWa, TeVer, WeWeth2}, and semi-classical states and singularly perturbed problems \cite{LiWe2,LiWe3,MoPeSq,Po}.

\paragraph{Notation.}

We consider non-autonomous systems of type
\begin{equation}\label{sysx}
\begin{cases}
- \Delta u_i + V_i(x) u_i = \sum_{j = 1}^d \beta_{ij}(x) u_j^2 u_i & \text{in $\Omega$} \\
u_1=\dots=u_d=0 &\text{on $\pa \Omega$}.
\end{cases}
\end{equation}
The possibility of considering non-constant potentials and coupling parameters is particularly interesting because, as explained in \cite{Timm}, it is coherent with the physical model of the Bose-Einstein condensation.

In what follows we state our basic assumptions.
They vary a little according to whether $\Omega$ is bounded or not. If $\Omega$ is a regular bounded domain of $\R^N$ with $N \le 3$, we suppose that 
\begin{equation}\tag{h0}\label{h0}
\begin{split}
&\text{$V_i, \beta_{ij} \in L^\infty(\Omega)$ for every $i$, and for every $(i,j)$, respectively}; \\
&\text{$\beta_{ij}=\beta_{ji}$ a.e. in $\Omega$ for every $i \neq j$;} \\
&\text{$V_i \ge \lambda_i> -\lambda_1(\Omega)$ a.e in $\Omega$, for every $i$;} \\
&\text{there exists $\mu_1,\dots,\mu_d>0$ such that $\beta_{ii} \ge \mu_i>0$, a.e. in $\Omega$, for every $i$; }
\end{split}
\end{equation}
Here $\lambda_1$ is the first eigenvalue of the Laplacian-Dirichlet in $\Omega$. If $\Omega=\R^N$ with $N =2,3$, then we require also that $V_i$  and $\beta_{ij}$ are radially symmetric with respect to the origin, and $V_i \ge \lambda_i>0$ a.e. in $\R^N$. In this latter situation, the boundary condition has to be replaced by $u_i \in H^1(\R^N)$ for every $i$.

From now on we consider $\Omega$ bounded; all the existence results can be trivially extended for problems in $\R^N$ (or in exterior radial domains) working in the set $H^1_r(\R^N)$ of the radially symmetric $H^1(\R^N)$ functions. This is possible in light of the Palais principle of symmetric criticality (see \cite{Pa}), and of the compactness of the embedding $H^1_r(\R^N) \hookrightarrow L^4(\R^N)$, valid for $N =2,3$. On the other hand, while in bounded domains we will often characterize positive solutions as least energy positive solutions, working in $H^1_r(\R^N)$ we can only find \emph{least energy radial positive solutions}, that is, radial positive solutions having minimal energy among all the radial positive solutions (in particular, we cannot exclude that $J(v_1,\dots,v_d) < \inf_{H^1_r(\R^N)} J$ for some non-radial function $(v_1,\dots,v_d)$). Concerning the celebrated system \eqref{systemdcomp}, we remind the interested reader to Theorem \ref{thm: RN} at the end of the introduction for an explicit statement in this setting. 

Let us introduce the main functional spaces and some notation.
\begin{itemize}
\item A bold face letters denotes a vector in some product space (which can be inferred by the context). In particular, $\mathbf{0}=(0,\dots,0)$ and $\mf{1}=(1,\dots,1)$ in some Euclidean space $\R^m$. The symbol $\cdot$ denotes the scalar product in an Euclidean space.
\item We consider the product space $\mathbb H:=(H_0^1(\Omega))^d$, with standard scalar product and norm.
\item The notation $|\cdot|_{p}$ is used for the $L^p(\Omega)$ norm, $1 \le p \le \infty$.
\item $\mf{B}(x):= (\beta_{ij}(x))_{i,j=1,\dots,d}$, and we refer to it as to the \emph{coupling matrix} of system \eqref{sysx}. 
\item We endow the Sobolev space $H_0^1(\Omega)$ with scalar products and norms
\[
\langle u, v \rangle_{V_i} =\langle u, v \rangle_i := \int_{\Omega} \left( \nabla u \cdot \nabla v + V_i(x) uv\right)  \quad \text{and} \quad \|u\|_{V_i}^2= \|u\|_i^2:=\langle u, u \rangle_i,
\]
for every $i=1,\dots,d$; in light of the Poincar\'e inequality and of the assumptions on $V_i$, these norms are equivalent to the standard one. 
\end{itemize}
\begin{remark}
In the paper we use both the above notations: when the choice of $V_i$ is fixed, we write simply $\langle \cdot, \cdot \rangle_i$ and $\|\cdot\|_i$; on the other hand, we consider also sequences of potentials $V_i^n$, for which the notation $\langle \cdot, \cdot \rangle_{V_i^n}$ and $\|\cdot\|_{V_i^n}$ is preferable. This observation holds also for the following definitions.
\end{remark}
\begin{itemize}
\item We let
\begin{equation}\label{def of S}
S:= \inf_{ u \in H_0^1(\Omega) \setminus \{0\} } \frac{ \int_{\Omega} |\nabla u|^2 }{|u|_4^2}.
\end{equation}
By Sobolev embedding, $S >0$. Moreover, for every $u \in H_0^1(\Omega)$ and for every $i$, we have $S|u|_4^2 \le \int_{\Omega} |\nabla u|^2 \le \|u\|_i^2$.
\item For an arbitrary $m \le d$, we say that a vector $\mf{a}=(a_0,\dots,a_m) \in \N^{m+1}$ is \emph{a $m$-decomposition of $d$} if
\[
0=a_0<a_1<\dots<a_{m-1}<a_m=d;
\]
given a $m$-decomposition $\mf{a}$ of $d$, we set, for $h=1,\dots,m$,
\[
\begin{split}
& I_h:= \{i \in  \{1,\dots,d\}:  a_{h-1} < i \le a_h \}, \\
&\mathcal{K}_1:= \left\{(i,j) \in I_h^2 \text{ for some $h=1,\dots,m$, with $i \neq j$}\right\}, \\ &\mathcal{K}_2 := \left\{(i,j) \in I_h \times I_k \text{ with $h \neq k$} \right\}. 
\end{split}
\]
\item For a given $m$-decomposition $\mf{a}$ of $d$, we introduce the $m \times m$ matrix
\begin{equation}\label{def matrice M}
\mf{M}(\mf{B},\mf{u}):= \left( \sum_{(i,j) \in I_h \times I_k} \int_{\Omega} \beta_{ij}(x) u_i^2 u_j^2\right)_{h,k=1,\dots,m},
\end{equation} 
where $\mf{u} \in \mathbb{H}$. Since $\beta_{ij}=\beta_{ji}$ for every $i \neq j$, this is a real symmetric matrix.
\item Let $\mf{a}$ be a $m$-decomposition of $d$, and let $\mf{u} \in \mathbb{H}$. We
set, for $h=1,\dots,m$,
\begin{equation}\label{def uh}
\mf{u}_h:= \left(u_{a_{h-1}+1},\dots,u_{a_h}\right) \in (H_0^1(\Omega))^{a_h-a_{h-1}}.
\end{equation}
The space $(H_0^1(\Omega))^{a_h-a_{h-1}}$ is naturally endowed with scalar product and norm
\[
\langle \mf{v}^1, \mf{v}^2 \rangle_{\mf{V}_h} =  \langle \mf{v}^1, \mf{v}^2 \rangle_h := \sum_{i \in I_h} \langle v_{i}^1, v_i^2 \rangle_i \quad \text{and} \quad \|\mf{v}\|_{\mf{V}_h}^2 = \|\mf{v}\|_h^2:=\langle \mf{v}, \mf{v} \rangle_h.
\]
\end{itemize}

\subsection{Main results}

First of all, the non-existence result for nonnegative solutions proved in \cite{BaWa,Si} can be trivially generalized in our setting, observing that any weak solution of \eqref{sysx} belongs to $(\mathcal{C}^{1,\alpha}(\Omega))^d$ for every $\alpha \in (0,1)$, and consequently the strong maximum principle applies.

\begin{proposition}\label{prop: non-ex}
Let $\mf{V},\mf{B}$ be as in \eqref{h0}. Assume that there exist $i, j \in \{1,\dots,d\}$ such that $V_i \ge V_j$, $\beta_{ik} \le \beta_{jk}$ a.e. in $\Omega$, for every $k=1,\dots,d$, and at least one of these inequalities is strict in a set of positive measure. Then a weak solution of \eqref{sysx} with $u_i, u_j \not \equiv 0$ and $u_i,u_j \ge 0$ does not exist.  
\end{proposition}

Let $\mf{a}$ be a $m$-decomposition of $d$, for some $m \le d$. Solutions of \eqref{sysx} are critical points of the energy functional
\[
\begin{split}
J_{\mf{V},\mf{B}}(\mf{u})  = J_{\mf{B}}(\mf{u}) &=\frac12 \int_{\Omega} \sum_{i=1}^d \left(|\nabla u_i|^2 + V_i(x) u_i^2\right) - \frac{1}{4} \int_{\Omega} \sum_{i,j=1}^d \beta_{ij}(x) u_i^2 u_j^2 \\
& = \frac{1}{2}\sum_{h=1}^m \|\mf{u}_h\|_h^2 - \frac{1}{4}  \mf{M}(\mf{B},\mf{u}) \mf{1} \cdot \mf{1},
\end{split}
\]
which is well defined in $\mathbb{H}$. We search for a solution of \eqref{sysx} as a constrained minimizer of $J_{\mf{B}}$ in a suitable subset of $\mathbb{H}$. We introduce the \emph{Nehari-type} set
\[
\mathcal{N}_{\mf{V},\mf{B}}=  \mathcal{N}_{\mf{B}}:= \left\{ \mf{u} \in \mathbb{H} \left| \begin{array}{l}
\|\mf{u}_h\|_h^2 \neq 0 \text{ and } \sum_{i \in I_h} \pa_i J_{\mf{B}}(\mf{u})u_i = 0 \\
\text{for every $h=1,\dots,m$}
\end{array}\right.\right\},
\]
where $\pa_i J_{\mf{B}}$ denotes the partial derivative of $J_{\mf{B}}$ with respect to $u_i$; moreover, we define
\[
\mathcal{E}_{\mf{B}} := \left\{ \mf{u} \in \mathbb{H}: \text{the matrix $\mf{M}(\mf{B},\mf{u})$ is strictly diagonally dominant} \right\},
\]
where we recall that a $m \times m$ matrix $A=(a_{ij})_{i,j}$ is strictly diagonally dominant if $|a_{ii}| > \sum_{j \neq i} |a_{ij}|$ for every $i=1,\dots,m$.

\begin{remark}\label{rem: similarity}
1) Note the similarity between the definition of $\mathcal{N}_{\mf{B}}$, and that of
\[
\mathcal{M}^*_{\mf{B}}= \left\{\mf{u} \neq \mf{0} \text{ and } \langle \nabla J_{\mf{B}}(\mf{u}),\mf{u} \rangle = 0\right\} \quad \text{or} \quad
\mathcal{M}_{\mf{B}}= \left\{ \begin{array}{l}
\|u_i\|_i^2 \neq 0 \text{ and } \pa_i J_{\mf{B}}(\mf{u})u_i = 0 \\
\text{for every $i=1,\dots,d$}
\end{array}\right\}.
\]
Actually, $\mathcal{N}_{\mf{B}}=\mathcal{M}^*_{\mf{B}}$, which is the Nehari manifold, in the degenerate situation of a $m$-decomposition of $d$ with $m=0$. Analogously, $\mathcal{N}_{\mf{B}}=\mathcal{M}_{\mf{B}}$ if $\mf{a}$ is the unique $d$-decomposition of $d$, that is, $\mf{a} = (0,1,\dots,d-1,d)$. \\
2) Regarding $\mathcal{E}_{\mf{B}}$, we recall that, as a consequence of the Gershgorin circle theorem, a strictly diagonally dominant, symmetric, real matrix with positive diagonal entries is positive definite. 
\end{remark}

The role of $\mathcal{N}_{\mf{B}}$ and $\mathcal{E}_{\mf{B}}$ is clarified by the following proposition.

\begin{proposition}\label{prop: natural constraint}
Let $\mf{u}^{\mf{B}} \in \mathcal{E}_{\mf{B}} \cap \mathcal{N}_{\mf{B}}$ be a constrained critical point of $J_{\mf{B}}$ restricted on $\mathcal{E}_{\mf{B}} \cap \mathcal{N}_{\mf{B}}$. Then $\mf{u}^{\mf{B}}$ is a free critical point of $J_{\mf{B}}$ in $\mathbb{H}$ with at least $m$ nontrivial components. In other worlds, $\mathcal{E}_{\mf{B}} \cap \mathcal{N}_{\mf{B}}$ is a natural constraint.
\end{proposition}

\begin{remark}
In light of the similarity between $\mathcal{N}_{\mf{B}}$ and the Nehari manifold $\mathcal{M}_{\mf{B}}$ and $\mathcal{M}^*_{\mf{B}}$ (see Remark \ref{rem: similarity}), it is not surprising that the minimization on $\mathcal{N}_{\mf{B}}$ gives a free critical point of $J_\mf{B}$ in $\mathbb{H}$; we refer to the recent contribution \cite{NoVe} for a general discussion concerning this type of constraints. On the other hand, the role of $\mathcal{E}_{\mf{B}}$ could appear quite mysterious at this stage. Such a role will be described in the forthcoming Remarks \ref{rem: geom constraint} (points 3) and 4)) and \ref{role of EB}.
\end{remark}

If we are interested in existence of semi-trivial solutions of \eqref{sysx} with at least $m$ nontrivial components, by Proposition \ref{prop: natural constraint} it is sufficient to find conditions on $\mf{V}$ and $\mf{B}$ which allow to prove that a minimizer for $J_{\mf{B}}$ restricted on $\mathcal{E}_{\mf{B}} \cap \mathcal{N}_{\mf{B}}$ does exist.

\begin{theorem}\label{thm:main1}
Let $d \ge 2$, let $\mf{a}$ be a $m$-decomposition of $d$ for some $m \le d$. Assume that \eqref{h0} holds, and let 
\begin{equation}\tag{h1}\label{h1}
\text{$\beta_{ij} \ge 0$ a.e. in $\Omega$ for every $(i,j) \in \mathcal{K}_1$}.
\end{equation}
There exists $\bar K>0$, depending on $\beta_{ii}, V_i$ and $\max_{(i,j) \in \mathcal{K}_1} |\beta_{ij}|_\infty$, such that if 
\begin{equation}\tag{h2}\label{small_coop_ass}
\text{$|\beta_{ij}^+|_\infty \le \bar K$ for every $(i,j) \in \mathcal{K}_2$}, 
\end{equation}
then there exists a nonnegative minimizer $\mf{u}^{\mf{B}}$ of $J_{\mf{B}}$ constrained on $\mathcal{N}_{\mf{B}} \cap \mathcal{E}_{\mf{B}}$. Consequently, $\mf{u}^{\mf{B}}$ is a solution of \eqref{sysx} with at least $m$ non-zero components. Furthermore, there exist $\overline{\gamma}$, $\underline{\gamma},\bar M>0$, depending on $\beta_{ii}, V_i$ and $\max_{(i,j) \in \mathcal{K}_1} |\beta_{ij}|_\infty$, such that 
\begin{equation}\label{estimates thm 1}
\sum_{i=1}^d \|u_i^{\mf{B}}\|_i^2 \le \overline{\gamma}, \quad \sum_{i \in I_h} |u_i^{\mf{B}}|_4^2 \ge \underline{\gamma}, \quad \sum_{i,j=1}^d \int_{\Omega}   \beta_{ij}^-(x) \left(u_i^{\mf{B}} u_j^{\mf{B}}\right)^2 \le \bar M.
\end{equation}
\end{theorem}
The proof proceeds by direct minimization. With respect to the known results in the literature, several complications arise by the fact that, in presence of simultaneous cooperation and competition, it is very difficult to control the fourth order term in the expression of $J_{\mf{B}}$. This is the main motivation for the introduction of the constraint $\mathcal{E}_{\mf{B}}$, and we refer the reader to the remarks of Section \ref{sec: semi-trivial} for more details. 

\begin{remark}
1) We can determine explicitly the dependence of $\overline{\gamma},\underline{\gamma}$, and $\bar M$ on $\beta_{ii},V_i$, and $\max_{(i,j) \in \mathcal{K}_1} |\beta_{ij}|_\infty$, see Lemmas \ref{lem: upper bounds}-\ref{lem: bound competition}. This will be crucial for the forthcoming results.\\
2) As already observed, any weak solution of \eqref{sysx} belongs to $(\mathcal{C}^{1,\alpha}(\Omega))^d$ for every $\alpha \in (0,1)$. Thus, by the maximum principle either $u_i^{\mf{B}} \equiv 0$, or $u_i^{\mf{B}}>0$ in $\Omega$.\\
3) Note that $\beta_{ij}$ can change sign for every $(i,j) \in \mathcal{K}_2$. This is particularly meaningful if one is interested in \eqref{sysx} as a model in population dynamics.
\end{remark}

We propose the following intuitive interpretation for Theorem \ref{thm:main1}. A $m$-decomposition of $d$ induces a separation of the $d$ components in $m$ different ``groups"; inside any group the interaction between the components is cooperative (see the \eqref{h1}). Theorem \ref{thm:main1} says that if the cooperation between any pair of different groups is sufficiently small (as expressed by the upper bound \eqref{small_coop_ass}), then it is possible to minimize $J_{\mf{B}}$ on the constraint $\mathcal{N}_{\mf{B}} \cap \mathcal{E}_{\mf{B}}$, obtaining, by Proposition \ref{prop: natural constraint}, a solution of \eqref{sysx} such that at least one component of each group is nontrivial. If compared with the known results for the $2$ components system \eqref{systemdcomp}, Theorem \ref{thm:main1} corresponds to what one obtain by minimizing the energy functional on the classical Nehari manifold; as observed above, this provides only a semi-trivial solution, and to obtain a nontrivial one careful level estimates are necessary. 

Before proceeding, we observe that there is one case in which Theorem \ref{thm:main1} immediately establishes the existence of a positive solution: let $\mf{a}=(0,1,\dots,d-1,d)$ be the unique $d$-decomposition of $d$; in this setting $\mathcal{K}_1= \emptyset$ and $\mathcal{K}_2=\left\{\beta_{ij}: i \neq j\right\}$, and an application of the previous result implies the following.


\begin{corollary}\label{corol:main2}
Let $d \ge 2$, and assume that \eqref{h0} holds. There exists $\bar K>0$, depending on $V_i$ and $\beta_{ii}$, such that if 
\begin{equation}\tag{h2*}\label{small_comp_ass_2}
\max_{i \neq j}  |\beta_{ij}^+|_{\infty}  < \bar K,
\end{equation}
then \eqref{sysx} has a positive solution $\mf{u}^{\mf{B}}$ which is a minimizer of $J_{\mf{B}}$ on $\mathcal{N}_{\mf{B}} \cap \mathcal{E}_{\mf{B}}$. 
\end{corollary}

\begin{remark}
To understand to what extent Corollary \ref{corol:main2} is new, let us consider for a moment the case of constant parameters $\lambda_i$ and $\beta_{ij}$. \\
1) For $\mf{a}=(0,1,\dots,d-1,d)$ it results $\mathcal{N}_{\mf{B}}= \mathcal{M}_{\mf{B}}$, see Remark \ref{rem: similarity}. The minimization of $J_{\mf{B}}$ and on $\mathcal{M}_{\mf{B}}$ has been already considered, in case of constant parameters $\beta_{ij}$, for systems of more than $2$ components, see \cite{LinWeiRNerr,LinWeiRN,LiWe2} (either in $\R^N$ on in bounded domains; the results can be trivially extended in the complementary situations by working, in case of $\Omega=\R^N$, in sets of radial functions). In the quoted papers the authors could obtain the existence of a minimal solution only for \emph{purely cooperative systems with small cooperation} (that is $\beta_{ij}>0$ for every $i \neq j$, small coupling parameters, and the positive definitiveness of the matrix $(\beta_{ij})_{i,j}$, see \cite{LinWeiRNerr,LinWeiRN}) or for \emph{purely competitive systems} (that is $\beta_{ij} \le 0$ for every $i \neq j$, see \cite{LiWe2}). The most relevant new feature of Corollary \ref{corol:main2} stays in the possibility of considering systems with mixed coupling parameters, without any constraint on the matrix $(\beta_{ij})$, obtaining the existence of a minimizer of $J_{\mf{B}}$ on $\mathcal{E}_{\mf{B}} \cap \mathcal{M}_{\mf{B}}$ under the only assumption that each $\beta_{ij}$ with $i \neq j$ \emph{has small positive part}. \\
2) The existence of at least one solution (neither necessarily positive, nor minimal) for systems with mixed coupling parameters has been proved in \cite{LiWa08}, but therein the authors assumed that $|\beta_{ij}|$ is small for every $i \neq j$. Corollary \ref{corol:main2} enlarges the set of mixed couplings for which a bound state of \eqref{systemdcomp} exists: it is sufficient to require that the positive part of $\beta_{ij}$ is sufficiently small for every $i \neq j$. If we are interested in positive solutions, we point out that such a result is optimal in many situations: for instance, if there exist $i$ and $j$ such that $\lambda_{i} \ge \lambda_{j}$ and $\beta_{ i i} < \beta_{ j  j}$, then whenever $\beta_{i j} \in \left( \beta_{ i i}, \beta_{ j  j}\right)$ a positive solution of \eqref{systemdcomp} cannot exist. \\
3) Another original feature of Corollary \ref{corol:main2} is represented by the possibility of considering non-constant coupling parameters, which \emph{can also change sign in $\Omega$}, provided their positive parts remain sufficiently small. As far as we know, this is new also for the $2$ component system.
\end{remark}

It is quite natural to ask whether or not the semi-trivial solutions we found are good candidates to be least energy positive solutions. We recall that a least energy positive solution $\mf{u}$ of \eqref{sysx} is a solution having minimal energy among all the positive solutions:
\[
J_{\mf{B}}(\mf{u}) =\inf\{J_{\mf{B}}(\mf{v}):  \nabla J_{\mf{B}}(\mf{v}) = 0 \text{ and } v_i > 0 \text{ for every $i$}\}.
\]

\begin{proposition}\label{prop: constraint E omitted}
Under the assumptions of Theorem \ref{thm:main1}, let us suppose that
\begin{equation}\label{relation pure competition}
\beta_{ij} \le 0 \qquad \text{a.e. in $\Omega$, for every $(i,j) \in \mathcal{K}_2$}.
\end{equation}
Then the infimum of $J_{\mf{V},\mf{B}}$ constrained on $\mathcal{N}_{\mf{V},\mf{B}}$ is achieved, any minimizer $\mf{u}^{\mf{B}}$ belongs to $\mathcal{E}_{\mf{B}}$, and in particular
\[
\inf_{\mathcal{N}_{\mf{V},\mf{B}}} J_{\mf{V},\mf{B}} = \inf_{\mathcal{N}_{\mf{V},\mf{B}} \cap \mathcal{E}_{\mf{B}}} J_{\mf{V},\mf{B}}.
\]
\end{proposition}
This means that, if the relations between different groups of components is purely competitive, then the constraint $\mathcal{E}_{\mf{B}}$ can be omitted, since it is somehow automatically included in $\mathcal{N}_{\mf{B}}$ (at least if we are interested in the research of minima). On the other hand, we point out that the role of $\mathcal{E}_{\mf{B}}$ is crucial in obtaining Theorem \ref{thm:main1} without assumption \eqref{relation pure competition} and, in particular, Corollary \ref{corol:main2}.

In what follows we assume that \eqref{relation pure competition} is satisfied, and, to obtain further existence results, we aim at finding suitable assumptions on the data $\mf{V}$ and $\mf{B}$ implying that
\begin{equation}\label{th variazionale}
\inf_{\mathcal{N}_{\mf{V},\mf{B}}}  J_{\mf{V},\mf{B}} <
\inf\left\{J_{\mf{V},\mf{B}}(\mf{u}) \left| \begin{array}{l}
\mf{u} \in \mathcal{N}_{\mf{V},\mf{B}}
\text{ and there exists} \\
 \text{$i =1,\dots,d$ such that $u_i \equiv 0$}
\end{array} \right.\right\}.
\end{equation}
If this condition holds, by minimality the solution $\mf{u}^{\mf{B}}$ found in Theorem \ref{thm:main1} has all positive components. In particular, this implies by Proposition \ref{prop: constraint E omitted} that it is a least energy positive solution of \eqref{sysx}.


\begin{theorem}\label{thm:main4}
Let $d \ge 2$, let $\mf{a}$ be a $m$-decomposition of $d$ for some $m \le d$. Let $\beta_{ii}$ as in assumption \eqref{h0}. For $h=1,\dots,m$, let $\tilde V_h, \tilde \beta_h \in L^\infty(\Omega)$ be such that 
\begin{equation}\tag{h3}\label{h3}
\text{$\tilde V_h \ge 0 \quad $ and $\quad \displaystyle \tilde \beta_h > \max_{i \in I_h} |\beta_{ii}|_\infty \quad $ a.e. in $\Omega$}.
\end{equation}
There exist $\delta,b>0$, depending on $\beta_{ii}, \tilde V_h, \tilde \beta_h$, such that if for every $h$ it results
\begin{equation}\tag{h4}\label{h4}
\begin{split}
&\text{$V_i \ge 0$ a.e. in $\Omega$ and $|V_i - \tilde V_h|_\infty < \delta$ for every $i \in I_h$, for every $h$},\\
&\text{$\beta_{ij} \ge 0$ a.e. in $\Omega$ and $|\beta_{ij}-\tilde \beta_h|_\infty < \delta$ for every $(i,j) \in I_h^2$ with $i \neq j$},
\end{split}
\end{equation}
and
\begin{equation}\tag{h5}\label{h5}
\text{$\beta_{ij} \le - b$ a.e. in $\Omega$ for every $(i,j) \in \mathcal{K}_2$},
\end{equation}
then \eqref{th variazionale} holds, and consequently system \eqref{sysx} has a least energy positive solution $\mf{u}^{\mf{B}}$.
\end{theorem}

For the proof, we have to suppose that the potentials and the coupling parameters which relate components of the same group are close together, see the \eqref{h4}. This assumption can be dropped if inside any group there are at most two components.

\begin{theorem}\label{thm:theorem 6}
Let $\mf{a}$ be a $m$-decomposition of $d$ such that $a_h-a_{h-1} \le 2$ for every $h=1,\dots,m$. Under \eqref{h0}, assume that
\begin{equation}\tag{h6}\label{ass strong coop 2}
\beta_{ij} > C_{ij} |\beta_{ii}|_\infty  \quad  \text{a.e. in $\Omega$, for every $(i,j) \in \mathcal{K}_1$}.
\end{equation}
Here $C_{ij} \ge 1$ is the best constant such that the inequality $\| u\|_i \le C \|u\|_j$ holds for every $u \in H_0^1(\Omega)$. Then there exists $b'>0$, depending on $V_i$, $\beta_{ii}$ and $\beta_{ij}$ with $(i,j) \in \mathcal{K}_1$, such that if
\begin{equation}\tag{h7}\label{h7}
\text{$\beta_{ij} \le - b'$ a.e. in $\Omega$, for every $(i,j) \in \mathcal{K}_2$},
\end{equation}
then \eqref{th variazionale} holds. As a consequence, system \eqref{sysx} has a least energy positive solution $\mf{u}^{\mf{B}}$.
\end{theorem}

\begin{remark}
1) Concerning assumptions \eqref{small_comp_ass_2}-\eqref{h7}, we can observe that, by Proposition \ref{prop: non-ex} and by thinking at the known results in the literature, in order to find a positive solution of \eqref{sysx} it is quite natural to suppose that $\beta_{ij}$ is either sufficiently large, or sufficiently small, with respect to both $\beta_{ii}$ and $\beta_{jj}$. \\
2) In light also of Theorem \ref{thm:theorem 6}, we do not believe that assumption \eqref{h4} is necessary. However, also for purely cooperative systems of type \eqref{systemdcomp} with many components ($d \ge 3$) the existence of a completely nontrivial bound state requires strong restriction on the data: we refer to Theorem 2.1, and in particular to Remark 2.2 and Corollary 2.3 in \cite{LiWa10}, where assumptions similar to \eqref{h4} are considered. \\
3) The proof of Theorem \ref{thm:main4} can be straightforwardly employed to obtain existence results also for purely cooperative cases, leading to results similar to those of \cite{LiWa10} (let $\mf{a}$ be a $0$-decomposition of $d$, so that $\mathcal{N}_{\mf{B}}$ is simply the Nehari manifold). In this setting, the advantage is that our method can be applied also for non-autonomous problems, both in bounded domains and in $\R^N$.
\end{remark}

A crucial intermediate step between Theorems \ref{thm:main1} and \ref{thm:main4} is the description of the segregation phenomena which involve solutions to \eqref{sysx}. 
Let $\mf{a}$ be a $m$-decomposition of $d$, and let us consider sequences $(\mf{V}^n) \subset  (L^\infty(\Omega))^d$, $(\mf{B}^n) \subset (L^\infty(\Omega))^{d^2}$, with $\mf{B}^n$ symmetric for every $n$, such that \eqref{h0} and \eqref{h1} holds, and 
$\beta_{ij}^n \le 0$ a.e. in $\Omega$, for every $(i,j) \in \mathcal{K}_2$.
Let $\mf{V}^\infty$ and $\mf{B}^\infty$ satisfying \eqref{h0} and \eqref{h1} as well. 
Let us suppose that $\mf{V}^n \to \mf{V}^\infty$ in $(L^\infty(\Omega))^d$, $\beta_{ii}^n \to \beta_{ii}^\infty$ in $L^\infty(\Omega)$ for every $i$, $\beta_{ij}^n \to \beta_{ij}^\infty$ in $L^\infty(\Omega)$ for every $(i,j) \in \mathcal{K}_1$, and 
\begin{equation}\label{ass segr}
\beta_{ij}^n \to -\infty \quad \text{in $L^\infty(\Omega)$ as $n \to \infty$, for every $(i,j) \in \mathcal{K}_2$}.
\end{equation}

By Theorem \ref{thm:main1} and Proposition \ref{prop: constraint E omitted}, for every $n$ there exists a minimizer $\mf{u}^n$ for $J_{\mf{V}^n,\mf{B}^n}$ on $\mathcal{N}_{\mf{V}^n,\mf{B}^n}$. Thus we obtain a sequence $(\mf{u}^n)$ of solutions of \eqref{sysx} with potentials $\mf{V}^n$ and coupling matrix $\mf{B}^n$, with the corresponding estimates \eqref{estimates thm 1}. By the explicit expression of $\overline{\gamma}, \underline{\gamma}, \bar M$, it is immediate to check that such estimates are uniform in $n$:
\begin{equation}\label{stime al limite}
\sum_{i=1}^d \int_{\Omega} |\nabla u_i^n|^2 \le \overline{\gamma}^\infty, \quad \sum_{i\in I_h} |u_i^n|_4^2 \ge \underline{\gamma}^\infty, \quad \sum_{i,j=1}^d \int_{\Omega} (\beta_{ij}^n)^- (u_i^n u_j^n)^2 \le \bar M^{\infty}, 
\end{equation}
where $\overline{\gamma}^{\infty}, \underline{\gamma}^{\infty}, \bar M^{\infty}$ are independent of $n$. As a consequence, up to a subsequence $\mf{u}^n \wc \mf{u}^\infty$ in $\mathbb{H}$, and 
\[
u^\infty_i u_j^\infty = \lim_{n \to \infty} u_i^n u_j^n = 0 \quad \text{a.e. in $\Omega$, for every $(i,j) \in \mathcal{K}_2$}. 
\]
This means that if two components $u_i^n$ and $u_j^n$ belong to different teams, then they tend to have disjoint support in the limit as $n \to \infty$; this phenomenon, called \emph{phase separation} or \emph{segregation}, can be described with more accuracy. Let us introduce, for an arbitrary $\mf{u} \in \mathbb{H}$, the $m \times m$ diagonal matrix
\begin{equation}\label{def Minfty}
\left(\mf{M}_\infty(\mf{u})\right)_{hk}:= \begin{cases} \sum_{(i,j) \in I_h^2} \int_{\Omega} \beta_{ij}^\infty(x) u_i^2 u_j^2 & \text{if $h=k$} \\ 0 & \text{if $h \neq k$}, \end{cases}
\end{equation}
and the functional
\begin{equation}\label{limit functional}
\begin{split}
J_\infty(\mf{u}) &:= \int_{\Omega} \left[\frac{1}{2} \sum_{i=1}^d \left(|\nabla u_i|^2 + V_i^\infty(x) u_i^2\right) - \frac{1}{4}\sum_{h=1}^m \sum_{(i,j) \in I_h^2} \beta_{ij}^\infty(x) u_i^2 u_j^2 \right]\\
& = \frac{1}{2}\sum_{h=m}^d \|\mf{u}_h\|_{\mf{V}_h^\infty}^2 - \frac{1}{4} \mf{M}_\infty(\mf{u}) \mf{1} \cdot \mf{1}.
\end{split}
\end{equation}
Correspondingly, we introduce the limit Nehari-type manifold
\begin{equation}\label{Nehari limite}
\mathcal{N}_{\infty}:= \left\{ \mf{u} \in \mathbb{H}\left| \begin{array}{l}
\|\mf{u}_h\|_{\mf{V}_h^\infty}^2 \neq 0 \text{ and } \sum_{i \in I_h} \pa_i J_{\infty}(\mf{u})u_i = 0 \\
\text{for every $h=1,\dots,m$}
\end{array}\right.\right\}.
\end{equation}

\begin{theorem}\label{thm:main3}
Up to a  subsequence, $(\mf{u}^n)$ converges strongly in $\mathbb{H}$ to a limiting profile $\mf{u}^\infty$, which achieves
\begin{equation}\label{var char of limit}
\inf\left\{ J_\infty(\mf{u}) \left| \begin{array}{l}
\mf{u} \in \mathcal{N}_{\infty} \text{ and } \int_{\Omega} u_i^2 u_j^2 = 0\\
\text{for every $(i,j) \in \mathcal{K}_2$}
\end{array}\right.\right\} .
\end{equation}
Furthermore, the estimates \eqref{stime al limite} hold true for $\mf{u}^\infty$, and 
\[
\lim_{n \to \infty}  \sum_{(i,j) \in \mathcal{K}_2} \int_{\Omega}  \left(\beta_{ij}^n\right)^- \left(u_i^n u_j^n\right)^2 =0.
\]
\end{theorem}

\begin{remark}
1) Similar results have been proved for completely competitive systems in \cite{CoTeVe2002,CoTeVe2003} (which considered ``minimal" solutions of slightly different systems), and in \cite{NoTaTeVe, WeWeth} (which considered uniformly bounded family of solutions). As far as we know, the unique other contribution studying phase separation in a not completely competitive system is \cite{CafLin}, where, however, it is assumed that $\beta_{ij} \le 0$ for every $i \neq j$ (this means that if two components do not compete, they do not interact at all). \\
2) The theorem works for both nonnegative and sign-changing solutions.
\end{remark}

Theorem \ref{thm:main3} says that if the relations between different groups is completely competitive, and the competition becomes stronger and stronger (as expressed by the \eqref{ass segr}), then pairs of components belonging to different groups tend to segregate, and the segregation occurs in such a way that the variational characterization of $\mf{u}^n$ passes to the limit: indeed, $\mf{u}^n \to \mf{u}^\infty$ in $\mathbb{H}$, where $\mf{u}^\infty$ is a minimizer for problem \eqref{var char of limit}. We point out that in this limit problem components belonging to different groups ``do not see" each other. This is essential to relate Theorems \ref{thm:main1} and \ref{thm:main4}: indeed this fact suggests that, if the competition between different groups is sufficiently strong, then system \eqref{sysx} can be considered as a perturbation of $m$ uncoupled systems of nonlinear equations with pure cooperation (recall that $\beta_{ij} \ge 0$ for every $(i,j) \in \mathcal{K}_1$). Such an idea shall be rigorously developed in the proof of Theorem \ref{thm:main4}: therein, we analyse the constrained second differential of $J_{\mf{B}}$ on $\mathcal{N}_{\mf{B}}$ evaluated in a minimum point $\mf{u}^\mf{B}$. Provided the competition between different groups is sufficiently strong (so that, as observed above, the problem is substantially uncoupled), we can prove that if either \eqref{h3} or \eqref{ass strong coop 2} are satisfied, then any minimizer of $J_{\mf{B}}$ on $\mathcal{N}_{\mf{B}}$ cannot have some zero components, otherwise the constrained second differential is not positive definite as a quadratic operator on the tangent space of $\mathcal{N}_{\mf{B}}$ in $\mf{u}^{\mf{B}}$. A similar analysis has been carried on for the $2$ components system \eqref{systemdcomp} in \cite{AmColor}.

\medskip

In light of the great efforts which have been devoted to the research of sufficient conditions on $\lambda_1,\dots, \lambda_d$ and $(\beta_{ij})$ for the existence of a nontrivial solution to the autonomous system \eqref{systemdcomp} in $\R^N$, we conclude this preliminary section with the following achievements, which are straightforward consequences of Corollary \ref{corol:main2}, Theorems \ref{thm:main4} and \ref{thm:theorem 6}. We recall that a least energy radial positive solution is a radial positive solution having minimal energy among all the radial positive solutions.

\begin{theorem}\label{thm: RN}
Let $d \ge 2$. Let $\lambda_1,\dots,\lambda_d \in \R$ and $(\beta_{ij})_{ij} \in \R^{d^2}$ such that $\lambda_i,\beta_{ii}>0$ for every $i$, and $\beta_{ij}=\beta_{ji}$ for every $i \neq j$.\\
($i$) There exists $K>0$ such that if $\beta_{ij} \le K$ for every $i \neq j$, then \eqref{systemdcomp} has a positive bound state which is radially symmetric with respect to the origin.\\
($ii$) let $\mf{a}$ be a $m$-decomposition of $d$. Let $\lambda_h>0$ for every $h=1,\dots,m$ and $\tilde \beta_h > \beta_{ii}$ for every $i \in I_h$ and $h=1,\dots,m$. There exist $\delta>0$ sufficiently small and $b>0$ sufficiently large such that, if 
\begin{align*}
|\lambda_i -\tilde \lambda_h|<\delta \quad \text{for every $i \in I_h$, for $h=1,\dots,m$}, \\
|\beta_{ij}- \tilde \beta_h|<\delta \quad \text{for every $(i,j) \in I_h^2$, $h=1,\dots,m$},\\
\beta_{ij}<-b \quad \text{for every $(i,j) \in \mathcal{K}_2$},
\end{align*}
then \eqref{systemdcomp} has a least energy radial positive solution.\\
($iii$) let $\mf{a}$ be a $m$-decomposition of $d$. Let us suppose that $a_h-a_{h-1}\le 2$ for every $h=1,\dots,m$; suppose also that 
\[
\beta_{ij} > C_{ij} \beta_{ii} \quad \text{for every $i,j \in I_h^2$, $h=1,\dots,m$},
\]   
where $C_{ij}$ is the best constant such that 
\[
\int_{\R^N} |\nabla u|^2+ \lambda_i u^2 \le C_{ij} \int_{\R^N} |\nabla u|^2+ \lambda_j u^2 \quad \text{for every $u \in H^1(\R^N)$}. 
\]
There exists $b'>0$ sufficiently large such that, if $\beta_{ij}<-b'$ for every $(i,j) \in \mathcal{K}_2$, then \eqref{systemdcomp} has a least energy radial positive solution. 
\end{theorem}  

\begin{corollary}
Let $\mf{a}$ be a $m$-decomposition of $d$. Let $\lambda_h>0$ for every $h=1,\dots,m$ and $\tilde \beta_h > \beta_{ii}$ for every $i \in I_h$ and $h=1,\dots,m$. There exists $b'>0$ sufficiently large such that, if $\beta_{ij}<-b'$ for every $(i,j) \in \mathcal{K}_2$, then \eqref{systemdcomp} has a least energy radial positive solution.
\end{corollary}

\paragraph{Structure of the paper.}
The proof of Propositions \ref{prop: natural constraint}, \ref{prop: constraint E omitted} and of Theorem \ref{thm:main1} is the object of the first section. The analysis of the segregation phenomena, Theorem \ref{thm:main3}, is faced in Section \ref{sec: segregation}, and the proof of Theorems \ref{thm:main4} and \ref{thm:theorem 6} is given is Section \ref{sec: strong}. Finally, Section \ref{sec: further} contains further results and comments.

\begin{remark}
After this work has been submitted, the author learned that the existence of positive solutions for the autonomous system \eqref{systemdcomp} with mixed couplings has been studied (independently) also in \cite{SaWa}. We think that is interesting to compare the results of our contribution and those of \cite{SaWa} when considering \eqref{systemdcomp} with $d=3$. At first, we point out that thanks to our Corollary \ref{corol:main2} the existence of a positive solution (not necessarily of least energy) is proved under the assumption $-\infty<\beta_{ij} \le K$ for every $i \neq j$, with $ K>0$ small enough (no conditions on the signs of $\beta_{ij}$ are necessary). Secondly, we assume that $\beta_{12}>0$ and $\beta_{13},\beta_{23} \le 0$. By Theorem 0.1 in \cite{SaWa}, there exists $\beta_* \gg 1$ such that if $\beta > \beta_*$, then \eqref{systemdcomp} has a least energy positive solution. By Theorem \ref{thm:main4} here, the same conclusion holds provided $\beta_{12} \ge C \max\{\beta_{11},\beta_{22}\}$ and $\beta_{13},\beta_{23} < -b'$, with $b' \gg 1$. Note that, since $\beta_*$ and $b'$ has to be though as arbitrarily large quantities, the two results are complementary and none of them can be seen as a particular case of the other. 
\end{remark}

\section{Semi-trivial solutions for systems with small cooperation between different ``groups"}\label{sec: semi-trivial}

This section is devoted to the proof of Proposition \ref{prop: natural constraint} and \ref{prop: constraint E omitted}, and of Theorem \ref{thm:main1}. For the reader convenience, we recall the assumptions we are considering: let $\Omega$ be a regular bounded domain of $\R^N$ with $N \le 3$, let $d \ge 2$, and let $\mf{a}$ be a $m$-decomposition of $d$ for some $m \le d$. We choose $\mf{V}$ and $\mf{B}$ as in \eqref{h0} and \eqref{h1}. Since in what follows the choice of $\mf{V}$ is fixed, we use the simplified notation $J_{\mf{B}}, \mathcal{N}_{\mf{B}},\dots$.

We search for a semi-trivial solution of \eqref{sysx} as a minimizer of $J_{\mf{B}}$ restricted on the intersection $\mathcal{N}_{\mf{B}} \cap \mathcal{E}_{\mf{B}}$. Thus, it is relevant to understand the geometry of the constraint. To do this, we explicitly write down the equations which define $\mathcal{N}_{\mf{B}}$ and $\mathcal{E}_{\mf{B}}$: $\mf{u} \in \mathcal{N}_{\mf{B}}$ if $\|\mf{u}_h\|_h^2>0$ and
\begin{equation}\label{eq nehari squadra}
\sum_{i \in I_h} \| u_i\|_i^2 = \sum_{k=1}^m \ \sum_{(i,j) \in I_h \times I_k} \int_{\Omega} \beta_{ij}(x) u_i^2 u_j^2 \quad \Longleftrightarrow \quad \|\mf{u}_h\|_h^2 = \sum_{k=1}^m \mf{M}(\mf{B},\mf{u})_{hk}
\end{equation}
for every $h=1,\dots,m$. Also, $\mf{u} \in \mathcal{E}_{\mf{B}}$ if $\mf{M}(\mf{B},\mf{u})_{hh} > \sum_{k \neq h} |\mf{M}(\mf{B},\mf{u})_{hk}|$, that is, 
\begin{equation}\label{eq open cone}
\sum_{(i,j) \in I_h^2} \int_{\Omega} \beta_{ij}(x) u_i^2 u_j^2 
> \sum_{k \neq h} \left| \sum_{(i,j) \in I_h \times I_k} \int_{\Omega} \beta_{ij}(x) u_i^2 u_j^2  \right| 
\end{equation}
for every $h=1,\dots,m$. We recall that $\mf{M}(\mf{B},\mf{u})$ and $\mf{u}_h$ have been defined by \eqref{def matrice M} and \eqref{def uh}, respectively.

\begin{remark}\label{rem: geom constraint}
1) The set $\mathcal{N}_{\mf{B}}$ is defined by a system of inequalities plus a system of equations $G_{\mf{B},h}(\mf{u}) = 0$, where
\begin{equation}\label{eq per Nehari nome}
G_{\mf{B},h}(\mf{u}):= \| \mf{u}_h\|_h^2 - \sum_{k=1}^m \mf{M}(\mf{B},\mf{u})_{hk}.
\end{equation}
For every $\varphi \in \mathbb{H}$ it results
\begin{equation}\label{diff constraint}
\langle \nabla G_{\mf{B},h}(\mf{u}), \mf{\varphi} \rangle = 2 \langle \mf{u}_h, \mf{\varphi}_h \rangle_h - 2\sum_{k=1}^m \ \sum_{(i,j) \in I_h \times I_k} \int_{\Omega} \beta_{ij}(x) u_i u_j (u_i \varphi_j + \varphi_i u_j).
\end{equation}
If $\mf{u} \in \mathcal{N}_{\mf{B}}$, then
$\langle \nabla G_{\mf{B},h}(\mf{u}), \mf{u} \rangle = - 2 \| \mf{u}_h\|_h^2 <0$; it follows that any $G_{\mf{B},h}$ defines, locally, a smooth manifold of codimension $1$ in $\mathbb{H}$. Let now $\mf{u} \in \mathcal{E}_{\mf{B}} \cap \mathcal{N}_{\mf{B}}$. We claim that in a neighbourhood of $\mf{u}$ the set $\mathcal{N}_{\mf{B}}$ defines a smooth manifold of codimension $m$ in $\mathbb{H}$. To verify this, we have to show that the differential $(d G_{\mf{B},1}(\mf{u}),\dots,d G_{\mf{B},m}(\mf{u}))$ is surjective as linear operator $\mathbb{H} \to \R^m$. By \eqref{eq per Nehari nome} and \eqref{diff constraint}, we have for $(t_1,\dots,t_m) \in \R^m$ and for every $h=1,\dots,m$
\[
\langle \nabla G_{\mf{B},h}(\mf{u}), (t_1 \mf{u}_1,\dots,t_m \mf{u}_m) \rangle = -2 \sum_{k=1}^m \mf{M}(\mf{B},\mf{u})_{hk} t_k.
\]
Since $\mf{u} \in \mathcal{E}_{\mf{B}}$, the matrix $\mf{M}(\mf{B},\mf{u})$ is non-singular and the claim is proved.

2) Due to the compact embedding $H_0^1(\Omega) \hookrightarrow  L^4(\Omega)$ and equation \eqref{eq open cone}, it is not difficult to check that $\mathcal{E}_{\mf{B}}$ is open in $\mathbb{H}$. 

3) The intersection $\mathcal{N}_{\mf{B}} \cap \mathcal{E}_{\mf{B}}$ is not empty. Actually, it is possible to prove that, given $\bar V_i$ and $\bar \beta_{ii}$ as in assumption \eqref{h0}, it results
\[
\bigcap \left\{ \mathcal{N}_{\mf{V},\mf{B}} \cap \mathcal{E}_{\mf{B}}\left| \begin{array}{l}
\text{$\mf{V}$ and $\mf{B}$ are such that $V_i = \bar V_i$} \\
\text{and } \beta_{ii}= \bar \beta_{ii} \text{ for every $i=1,\dots,d$}
\end{array} \right. \right\}\neq \emptyset.
\]
To show this, for any $\mf{u} \in \mathbb{H}$ we introduce a function $\Psi_{\mf{B}, \mf{u}}: \overline{(\R_+)^m} \to \R$ defined by
\begin{equation}\label{exp Psi}
\begin{split}
\Psi_{\mf{B}, \mf{u}}(\mf{t}) &:= J_{\mf{B}} \left( \sqrt{t_1} \mf{u}_1, \dots, \sqrt{t_h} \mf{u}_h, \dots, \sqrt{t_m} \mf{u}_m \right) \\
&  = \frac{1}{2} \sum_{h=1}^m  \|\mf{u}_h\|_h^2 t_h - \frac{1}{4} \mf{M}(\mf{B},\mf{u}) \mf{t} \cdot \mf{t}.
\end{split}
\end{equation}
Note that if $\mf{t} \in (\R_+)^m$ is a critical point of $\Psi_{\mf{B}, \mf{u}}$, then 
$( \sqrt{t_1}\mf{u}_{1},\dots, \sqrt{t_m} \mf{u}_m) \in \mathcal{N}_{\mf{B}}$. Now, let $\tilde{\mf{u}}$ be such that $\tilde u_i \not \equiv 0$ for every $i$, and $\tilde u_i \tilde u_j \equiv 0$ for every $i \neq j$. The matrix $\mf{M}(\mf{B},\tilde{\mf{u}})$ is then a diagonal matrix with strictly positive diagonal entries, which does not depend on the particular choice of $\mf{B}$ but only on $\bar \beta_{ii}$; hence
\[
\tilde{\mf{u}} \in \bigcap \left\{ \mathcal{E}_{\mf{B}} \left| \begin{array}{l}
\text{$\mf{V}$ and $\mf{B}$ are such that $V_i = \bar V_i$} \\
\text{and } \beta_{ii}= \bar \beta_{ii} \text{ for every $i=1,\dots,d$}
\end{array} \right. \right\}.
\]
Furthermore, one can easily check that $\tilde{\mf{t}}$, defined by $\tilde t_h =  \| \tilde{\mf{u}}_h\|_h^2  / \mf{M}(\mf{B}, \tilde{\mf{u}})_{hh} >0$ for every $h=1, \dots, m$ is a critical point of $\Psi_{\mf{B},\tilde{\mf{u}}}$ in $(\R_+)^m$ for every $\mf{V}$ and $\mf{B}$ such that $V_i= \bar V_i$ and $\beta_{ii}= \bar \beta_{ii}$; thus
\[
(\sqrt{ \tilde t_1}\tilde{\mf{u}}_1,\dots, \sqrt{\tilde t_m} 	\tilde{\mf{u}}_m) \in \bigcap \left\{ \mathcal{N}_{\mf{V},\mf{B}} \cap \mathcal{E}_{\mf{B}}\left| \begin{array}{l}
\text{$\mf{V}$ and $\mf{B}$ are such that $V_i = \bar V_i$} \\
\text{and } \beta_{ii}= \bar \beta_{ii} \text{ for every $i=1,\dots,d$}
\end{array} \right. \right\}.
\]

4) If $\mf{u} \in \mathcal{N}_{\mf{B}} \cap \mathcal{E}_{\mf{B}}$, then $\Psi_{\mf{B}, \mf{u}}$ has the unique maximum point $\mf{1}$ in $\overline{(\R_+)^m}$. Indeed, 
by $\mf{u} \in \mathcal{N}_{\mf{B}}$ we immediately see that the point $\mf{1}$ is a critical point of $\Psi_{\mf{B}, \mf{u}}$. As function of $\mf{t}$, this is a polynomial of degree $2$, and the related quadratic form is negative definite since $\mf{u} \in \mathcal{E}_{\mf{B}}$. This implies that $\Psi_{\mf{B}, \mf{u}}$ has at most one inner critical point, which has to be a strict maximum.

5) If $\mf{u} \in \mathcal{N}_{\mf{B}}$, then 
\begin{equation}\label{func on Nehari}
J_{\mf{B}}(\mf{u}) = \frac{1}{4} \mf{M}(\mf{B},\mf{u}) \mf{1} \cdot\mf{1} = \frac{1}{4} \sum_{h =1}^m \|\mf{u}_h\|_h^2 >0.
\end{equation}
In particular, the functional $J_{\mf{B}}$ is bounded below on $\mathcal{N}_{\mf{B}}$.
\end{remark}

Having understood the properties of the constraint, Proposition \ref{prop: natural constraint} follows easily. 

\begin{proof}[Proof of Proposition \ref{prop: natural constraint}] 
As $\mf{u}^{\mf{B}} \in \mathcal{N}_{\mf{B}}$, it results $\|\mf{u}_h^{\mf{B}}\|_h^2 >0$, so that at least $m$ components of $\mf{u}^{\mf{B}}$ are nontrivial. The constraint $\mathcal{N}_{\mf{B}} \cap \mathcal{E}_{\mf{B}}$ is an open subset of $\mathcal{N}_{\mf{B}}$ in the topology of $\mathbb{H}$. Thus the function $\mf{u}^{\mf{B}}$ is an inner critical point of $J_{\mf{B}}$ in an open subset of $\mathcal{N}_{\mf{B}}$, and in particular it is a constrained critical point of $J_{\mf{B}}$ on $\mathcal{N}_{\mf{B}}$. We showed that in a neighbourhood of $\mf{u}^{\mf{B}} \in \mathcal{E}_{\mf{B}}$ this is a smooth manifold of codimension $m$, so that by the Lagrange multipliers rule there exist $\mu_1,\dots,\mu_m \in \R$ such that
\begin{equation}\label{constrained critical point}
\nabla J_{\mf{B}} (\mf{u}^{\mf{B}}) - \sum_{h=1}^m \mu_h \nabla G_{\mf{B},h}(\mf{u}^{\mf{B}}) = 0.
\end{equation}
Let us test the \eqref{constrained critical point} with $(0,\dots,0, \mf{u}_h^{\mf{B}},0 \dots, 0)$, for $h=1,\dots,m$: 
recalling expression \eqref{diff constraint}, and that $G_{\mf{B},h}(\mf{u}^{\mf{B}})=0$ for every $h$, we deduce
\begin{align*}
0 & = 2 \left( \|\mf{u}_h\|_h^2 - 2 \mf{M}(\mf{B},\mf{u}^{\mf{B}})_{hh}- \sum_{k \neq h} \mf{M}(\mf{B},\mf{u}^{\mf{B}})_{hk} \right) \mu_h - 2 \sum_{k \neq h} \mf{M}(\mf{B},\mf{u}^{\mf{B}})_{hk} \mu_k \\
& = -2 \sum_{k=1}^m \mf{M}(\mf{B},\mf{u}^{\mf{B}})_{hk} \mu_k,
\end{align*}
for every $h=1,\dots,m$. This means that $\mu_1,\dots,\mu_m$ are defined by the linear homogeneous system $\mf{M}(\mf{B},\mf{u}^{\mf{B}}) \mf{\mu} = \mf{0}$, which, in light of the fact that $\mf{u}^{\mf{B}} \in \mathcal{E}_{\mf{B}}$, has the unique solution $\mf{\mu}=\mf{0}$. Plugging into the \eqref{constrained critical point}, we see that $\nabla J_{\mf{B}}(\mf{u}^{\mf{B}}) = 0$, that is, $\mf{u}^{\mf{B}}$ is a free critical point of $J_{\mf{B}}$ in $\mathbb{H}$. 
\end{proof}

Let $\mf{V}$ and $\mf{B}$ satisfying the \eqref{h0} and \eqref{h1}. In what follows we aim at proving that, provided $\max_{(i,j) \in \mathcal{K}_2} |\beta_{ij}^+|_\infty$ is sufficiently small, the variational problem
\[
c_{\mf{B}} := \inf_{\mf{u} \in \mathcal{N}_{\mf{B}} \cap \mathcal{E}_{\mf{B}}} J_{\mf{B}}(\mf{u})
\]
admits a minimizer. Note that, in light of points 3) and 5) of Remark \ref{rem: geom constraint}, $J_{\mf{B}}$ is bounded from below on the constraint $\mathcal{N}_{\mf{B}} \cap \mathcal{E}_{\mf{B}}$, which is non-empty. Let $(\mf{u}^{\mf{B}}_n)\subset \mathcal{N}_{\mf{B}} \cap \mathcal{E}_{\mf{B}}$ be a minimizing sequence for $c_{\mf{B}}$. In light of the \eqref{eq nehari squadra}, \eqref{eq open cone}, and of the definition of $J_{\mf{B}}$, it is not restrictive to assume that $u_{i,n}^{\mf{B}} \ge 0$ a.e. in $\Omega$, for every $i=1,\dots,d$ and $n \in \N$. In the next three lemmas we prove some useful properties of $(\mf{u}^{\mf{B}}_n)$. Inside the proofs, where the choice of $\mf{B}$ is fixed, we simply write $(\mf{u}_n)$ to ease the notation.

\begin{lemma}\label{lem: upper bounds}
Let $(\mf{u}_n^{\mf{B}}) \subset \mathcal{N}_{\mf{B}} \cap \mathcal{E}_{\mf{B}}$ be a minimizing sequence for $c_{\mf{B}}$. Then there exists a universal constant $\bar C>0$ such that
\[
\sum_{i=1}^d \| u_{i,n}^{\mf{B}}\|_i^2  \le  \bar C\left( \frac{\left( 1+ \max_{i } |V_i|_\infty\right)^2}{\min_{i } \mu_i}+1\right) =: \overline{\gamma} \quad \text{and} \quad c_{\mf{B}} \le \overline{\gamma},
\]
provided $n$ is sufficiently large.
\end{lemma}
\begin{proof}
We take $\left( \sqrt{\tilde t_1}\tilde{\mf{u}}_1,\dots, \sqrt{\tilde t_m} \tilde{\mf{u}}_m\right) \in  \mathcal{N}_{\mf{B}} \cap \mathcal{E}_{\mf{B}}$  for every $\mf{B}$ sharing the same diagonal elements $\beta_{ii}$, and for a fixed choice of $\mf{V}$, as in point 3) of Remark \ref{rem: geom constraint}. Clearly, recalling the definition of $\tilde{t}_h$ and the expression of $J_{\mf{B}}$ on $\mathcal{N}_{\mf{B}}$, equation \eqref{func on Nehari}, it results
\begin{align*}
c_{\mf{B}} &\le J_{\mf{B}}\left( \sqrt{\tilde t_1}\tilde{\mf{u}}_1,\dots, \sqrt{\tilde t_m} \tilde{\mf{u}}_m\right)  = \frac{1}{4} \sum_{h=1}^m \|\tilde{\mf{u}}_h\|_h^2 \tilde t_h \\
&\le \frac{1}{4} \sum_{h=1}^m \frac{\left( 1+ \max_{i} |V_i|_\infty\right)^2}{\min_i \mu_i} \cdot \frac{ \left(\sum_{i \in I_h} \int_{\Omega} |\nabla \tilde{u}_i|^2+\tilde{u}_i^2\right)^2} {\sum_{i \in I_h}\int_{\Omega} \tilde u_i^4} \le \tilde C \frac{\left( 1+ \max_{i} |V_i|_\infty\right)^2}{\min_{i} \mu_i};
\end{align*}
here $\tilde C$ is a positive constant depending only by the choice of $\tilde{\mf{u}}$ as in Remark \ref{rem: geom constraint}, and is independent on $\mf{V}$ and $\mf{B}$. Now, if $(\mf{u}_n)$ is a minimizing sequence for $c_{\mf{B}}$, it is not restrictive to assume that $J_{\mf{B}}(\mf{u}_n) \le c_{\mf{B}}+1$ for every $n$, and the desired result follows.
\end{proof}


\begin{lemma}\label{lem: lower bound forte}
Let $(\mf{u}_n^{\mf{B}}) \subset \mathcal{N}_{\mf{B}} \cap \mathcal{E}_{\mf{B}}$ be a minimizing sequence for $c_{\mf{B}}$. Assume that there exists $0<K_1 < S^2/\overline{\gamma}$ such that $\max_{(i,j) \in \mathcal{K}_2}  |\beta_{ij}^+|_\infty  \le K_1$, where $S$ and $\overline{\gamma}$ have been defined in \eqref{def of S} and Lemma \ref{lem: upper bounds}, respectively; then 
\begin{align*}
\sum_{i \in I_h} |u_{i,n}^{\mf{B}}|_4^2 & 
\ge \frac{1}{\displaystyle d \left( \max_{(i,j)\in \mathcal{K}_1} |\beta_{ij}|_{\infty} + \max_{i} |\beta_{ii}|_{\infty}\right)}\left( S- \frac{\overline{\gamma}}{S} \max_{(i,j)\in \mathcal{K}_2}|\beta_{ij}^+|_{\infty}\right)=: \underline{\gamma}>0,
\end{align*}
for every $h=1,\dots,m$.
\end{lemma}
\begin{proof}
If $\mf{u} \in \mathcal{N}_{\mf{B}}$, then $G_{\mf{B},h}(\mf{u})=0$ and $\sum_{i \in I_h} |u_i|_4^2 >0$ for every $h$. Therefore, by routine arguments based on H\"older and Young inequalities, we deduce that
\begin{multline*}
S \sum_{i \in I_h} |u_{i,n}|_4^2 \le \|\mf{u}_{h,n}\|_h^2 = \sum_{k=1}^m\mf{M}(\mf{B},\mf{u}_n)_{hk} \le \sum_{k=1}^m \sum_{(i,j) \in I_h \times I_k} |\beta_{ij}^+|_\infty |u_{i,n}|_4^2 |u_{j,n}|_4^2  \\
 \le \sum_{(i,j)\in I_h^2} \frac{|\beta_{ij}^+|_{\infty}}{2}\left( |u_{i,n}|_4^4 + |u_{j,n}|_4^4\right) + \sum_{k \neq h} \sum_{(i,j) \in I_h \times I_k} |\beta_{ij}^+|_\infty |u_{i,n}|_4^2 |u_{j,n}|_4^2\\
\le \sum_{i \in I_h} \left( \sum_{j \in I_h} |\beta_{ij}|_{\infty} \right) |u_{i,n}|_4^4  + \frac{ \overline{\gamma}}{S} \left( \max_{(i,j)\in \mathcal{K}_2} |\beta_{ij}^+|_\infty\right)    \sum_{i \in I_h} |u_{i,n}|_4^2\\
\le d\left( \max_{(i,j)\in \mathcal{K}_1} |\beta_{ij}|_{\infty} + \max_{i } |\beta_{ii}|_{\infty}\right) \left(\sum_{i \in I_h} |u_{i,n}|_4^2\right)^2 + \frac{\overline{\gamma}}{S} \left( \max_{(i,j)\in \mathcal{K}_2}  |\beta_{ij}^+|_\infty \right) \sum_{i \in I_h} |u_{i,n}|_4^2,
\end{multline*} 
where we used the fact that, by Lemma \ref{lem: upper bounds},
\[
\sum_{k \neq h} \sum_{j \in I_k} |u_{j,n}|_4^2 \le \sum_{j=1}^d |u_{j,n}|_4^2\le \frac{1}{S} \sum_{i=1}^d \|u_{i,n}\|_i^2 \le \frac{\overline{\gamma}}{S}. \qedhere
\]
\end{proof}

We can also bound the competitive part of the interaction terms.

\begin{lemma}\label{lem: bound competition}
Let $(\mf{u}_n^{\mf{B}}) \subset \mathcal{N}_{\mf{B}} \cap \mathcal{E}_{\mf{B}}$ be a minimizing sequence for $c_{\mf{B}}$. It results
\[
\sum_{i,j=1}^d \int_{\Omega}   \beta_{ij}^-(x) \left(u_{i,n}^{\mf{B}} u_{j,n}^{\mf{B}} \right)^2 \le \frac{\overline{\gamma}^2}{S^2}\left( \max_{(i,j)\in \{1,\dots,d\}^2} |\beta_{ij}^+|_{\infty}\right)=: \bar M.
\]
\end{lemma}
\begin{proof}
By the positivity of $J_{\mf{B}}$ on the constraint, we have 
\begin{multline*}
 \sum_{i,j=1}^d\int_{\Omega}    \beta_{ij}^- u_{i,n}^2 u_{j,n}^2  \le   \sum_{i,j=1}^d  \int_{\Omega}  \beta_{ij}^+ u_{i,n}^2 u_{j,n}^2  \\
\le \left(\max_{(i,j) \in \{1,\dots,d\}^2 } |\beta_{ij}^+|_\infty\right) \left( \sum_{i=1}^d |u_{i,n}|_4^2\right) \left(\sum_{j=1}^d |u_{j,n}|_4^2\right)  \le \left( \max_{(i,j)\in \{1,\dots,d\}^2} |\beta_{ij}^+|_{\infty}\right) \left(\frac{\overline{\gamma}}{S}\right)^2.
\end{multline*}
Note that in the last inequality we used Lemma \ref{lem: upper bounds} to bound $\sum_{i=1}^d |u_{i,n}|_4^2$.
\end{proof}

Let us collect what we proved so far. For $\mf{V}$ and $\mf{B}$ satisfying assumptions \eqref{h0} and \eqref{h1}, we defined $J_{\mf{B}}$, $\mathcal{N}_{\mf{B}}$, and $\mathcal{E}_{\mf{B}}$, and considered the minimization problem $c_{\mf{B}} = \inf_{\mathcal{N}_{\mf{B}} \cap \mathcal{E}_{\mf{B}}} J_{\mf{B}}$. Let $(\mf{u}_n^{\mf{B}}) \subset \mathcal{N}_{\mf{B}} \cap \mathcal{E}_{\mf{B}}$ be a minimizing sequence for $J_{\mf{B}}$, where it is not restrictive to assume that $u_{i,n}^{\mf{B}} \ge 0$ a.e. in $\Omega$ for every $i$ and $n$. By Lemma \ref{lem: upper bounds}, there exists $\overline{\gamma}$, depending only on $V_i$ and on $\beta_{ii}$, such that
\[
\sum_{i=1}^d \| u_{i,n}^{\mf{B}}\|_i^2  \le \overline{\gamma} \quad \text{and} \quad c_{\mf{B}} \le \overline{\gamma}.
\]
Now, fixed a positive number $K_1<S^2/\overline{\gamma}$, we supposed that $\max_{(i,j) \in \mathcal{K}_2}  |\beta_{ij}^+|_\infty  \le K_1$. Under this assumption, by Lemmas \ref{lem: lower bound forte} and \ref{lem: bound competition} there exist $\underline \gamma, \bar{M}>0$, such that
\[
\sum_{i \in I_h} |u_{i,n}^{\mf{B}}|_4^2 \ge \underline{\gamma} \qquad h=1,\dots,m, \quad \text{and} \quad  \sum_{i,j=1}^d \int_{\Omega}   \beta_{ij}^-(x) \left(u_{i,n}^{\mf{B}} u_{j,n}^{\mf{B}} \right)^2 \le \bar M.
\]

To proceed, we observe that being $H_0^1(\Omega)$ a reflexive space, and thanks to the compactness of the Sobolev embedding $H_0^1(\Omega) \hookrightarrow L^4(\Omega)$, there exists $\mf{u}^{\mf{B}}$ such that, up to a subsequence, $\mf{u}_n^{\mf{B}} \wc \mf{u}^{\mf{B}}$ weakly in $\mathbb{H}$, $\mf{u}_n^{\mf{B}} \to \mf{u}^{\mf{B}}$ strongly in $(L^4(\Omega))^d$, and $\mf{u}_n^{\mf{B}} \to \mf{u}^{\mf{B}}$ a.e. in $\Omega$; in particular, $u_i^{\mf{B}} \ge 0$ a.e. in $\Omega$, for every $i=1,\dots,d$. Moreover, since $\mf{u}_n^{\mf{B}} \in \mathcal{N}_{\mf{B}}$ for every $n$, it results
\begin{equation}\label{inequ Nehari limite}
\|\mf{u}_h^{\mf{B}}\|_h^2 \le \sum_{k=1}^m \mf{M}(\mf{B},\mf{u}^{\mf{B}})_{hk},
\end{equation}
for $h=1,\dots,m$, and the bounds proved in Lemmas \ref{lem: upper bounds}-\ref{lem: bound competition} hold true for $\mf{u}^{\mf{B}}$.

We aim at proving that $\mf{u}^{\mf{B}} \in \mathcal{N}_{\mf{B}} \cap \mathcal{E}_{\mf{B}}$; this result requires several lemmas. Before proceeding, we point out that if $\mf{u}^{\mf{B}} \in \mathcal{N}_{\mf{B}} \cap \mathcal{E}_{\mf{B}}$, then by definition $c_{\mf{B}} \le J_{\mf{B}}(\mf{u}^{\mf{B}})$, and by weak lower semi-continuity $J_{\mf{B}}(\mf{u}^{\mf{B}}) \le \liminf_n J_{\mf{B}}(\mf{u}_n^{\mf{B}}) \le c_{\mf{B}}$, so that $\mf{u}^{\mf{B}}$ is a minimizer of $J_{\mf{B}}$ on $\mathcal{N}_{\mf{B}} \cap \mathcal{E}_{\mf{B}}$. 

\begin{lemma}\label{lem: limit in cone}
Let $0<K_2 < S^2/(2\overline{\gamma})$. If
\begin{equation}\label{ass_small_coop2}
\max_{(i,j) \in \mathcal{K}_2}  |\beta_{ij}^+|_\infty \le K_2,
\end{equation}
then $\mf{u}^{\mf{B}} \in \mathcal{E}_{\mf{B}}$.
\end{lemma}
\begin{proof}
We assume that $\mf{u}^{\mf{B}} \not \in \mathcal{E}_{\mf{B}}$, and we show that necessarily the quantity on the left hand side of \eqref{ass_small_coop2} is larger than or equal to $S^2/(2\overline{\gamma})$. Since $\mf{u}_n^{\mf{B}} \in \mathcal{E}_{\mf{B}}$ for every $n$, the $(L^4(\Omega))^d$ convergence of $\mf{u}_n^{\mf{B}}$ to $\mf{u}^{\mf{B}}$ implies that there exists $\bar h \in \{1,\dots,m\}$ such that
\begin{align*}
\mf{M}(\mf{B},\mf{u}^{\mf{B}})_{\bar h \bar h} &= \sum_{k \neq \bar h} |\mf{M}(\mf{B},\mf{u}^{\mf{B}})_{\bar h k} |  \le \sum_{k \neq \bar h} \ \sum_{(i,j) \in I_{\bar h} \times I_k} \int_{\Omega} \left(\beta_{ij}^+ + \beta_{ij}^- \right) \left( u_i^{\mf{B}} u_j^{\mf{B}} \right)^2.
\end{align*} 
Combining this with inequality \eqref{inequ Nehari limite}, we deduce
\begin{multline*}
S \sum_{i \in I_{\bar h}} |u_i^{\mf{B}}|_4^2  \le  \|\mf{u}_{\bar h}^{\mf{B}}\|_{\bar h}^2 \le \mf{M}(\mf{B},\mf{u}^{\mf{B}})_{\bar h \bar h}   + \sum_{ k \neq \bar h} \ \sum_{(i,j) \in I_{\bar h} \times I_k} \int_{\Omega} \left(\beta_{ij}^+ - \beta_{ij}^- \right) \left( u_i^{\mf{B}} u_j^{\mf{B}} \right)^2 \\
 \le 2 \sum_{ k \neq \bar h} \ \sum_{(i,j) \in I_{\bar h} \times I_k} \int_{\Omega} \beta_{ij}^+ \left( u_i^{\mf{B}} u_j^{\mf{B}} \right)^2  \le 2 \sum_{ k \neq \bar h} \ \sum_{(i,j) \in I_{\bar h} \times I_k}  |\beta_{ij}^+|_\infty |u_i^{\mf{B}}|_4^2 |u_j^{\mf{B}}|_4^2.
\end{multline*}
Now, using the uniform upper bound proved in Lemma \ref{lem: upper bounds}, which holds true for $\mf{u}^{\mf{B}}$ by weak lower semi-continuity, we infer
\begin{align*}
S \sum_{i \in I_{\bar h}} |u_i^{\mf{B}}|_4^2 & \le 2\left( \max_{(i,j) \in \mathcal{K}_2} |\beta_{ij}^+|_\infty \right) \left( \sum_{i \in I_{\bar h}} |u_i^{\mf{B}}|_4^2\right) \left(\sum_{ j \in \{1,\dots,d\} \setminus I_{\bar h}}    |u_j^{\mf{B}}|_4^2\right) \\
& \le \frac{2\overline{\gamma}}{S} \left(\max_{(i,j) \in \mathcal{K}_2}   |\beta_{ij}^+|_\infty\right)   \sum_{i \in I_{\bar h}} |u_i^{\mf{B}}|_4^2,
\end{align*}
which in turn implies $\max_{(i,j) \in \mathcal{K}_2}   |\beta_{ij}^+|_\infty   \ge S^2/(2 \overline{\gamma})$. 
\end{proof}

\begin{remark}\label{role of EB}
To show that $\mf{u}^{\mf{B}} \in \mathcal{N}_{\mf{B}}$, we study the auxiliary function $\Psi_{\mf{B}}:=\Psi_{\mf{B},\mf{u}^{\mf{B}}}$, where $\Psi_{\mf{B},\mf{u}}$ has been defined for any $\mf{u} \in \mathbb{H}$ in \eqref{exp Psi}. It is useful to recall what we observed in Remark \ref{rem: geom constraint}, at points 3) and 4): 
\begin{itemize}
\item if $\mf{1}$ is a critical point of $\Psi_{\mf{B},\mf{u}}$, then $\mf{u} \in \mathcal{N}_{\mf{B}}$;
\item if $\mf{u} \in \mathcal{N}_{\mf{B}} \cap \mathcal{E}_{\mf{B}}$, then $\Psi_{\mf{B}, \mf{u}}$ has the unique maximum point $\mf{1}$ in $\overline{(\R_+)^m}$.
\end{itemize}
Hence, we wish to prove that $\mf{1}$ is a critical point of $\Psi_{\mf{B}}$. An intermediate and difficult step consists in showing that $\Psi_{\mf{B}}$ has a unique maximum point in the inner of $(\R_+)^m$. We emphasize the fact that for an arbitrary $\mf{u} \in \mathbb{H}$, the function $\Psi_{\mf{B},\mf{u}}$ is not necessarily bounded above, and the maximization fails. This is the main motivation for the introduction of the constraint $\mathcal{E}_{\mf{B}}$: indeed, if $\mf{u} \in \mathcal{E}_{\mf{B}}$ it is not difficult to check that a maximum point does exist. 
\end{remark}

\begin{lemma}\label{lem: max raggiunto}
If $\mf{u} \in \mathcal{E}_{\mf{B}}$, then the supremum of $\Psi_{\mf{B},\mf{u}}$ in $\overline{(\R_+)^m}$ is achieved.
\end{lemma}
The proof is easy and we omit it.



By Lemmas \ref{lem: limit in cone} and \ref{lem: max raggiunto}, we deduce that under assumption \eqref{ass_small_coop2} there exists a maximum point $\mf{t}^{\mf{B}}$ for $\Psi_{\mf{B}}$. As $\Psi_{\mf{B}}(\mf{0}) = 0$, we deduce $\mf{t}^{\mf{B}} \neq \mf{0}$.

\begin{lemma}\label{lem: max interno}
There exists $0<K_3 \le K_2$, depending on $V_i,\beta_{ii}$ and $\max_{(i,j) \in \mathcal{K}_1} |\beta_{ij}|_\infty$, such that if $\max_{(i,j) \in \mathcal{K}_2} |\beta_{ij}^+|_\infty  \le K_3$, then any maximum point $\mf{t}^{\mf{B}}$ of $\Psi_{\mf{B}}$ is such that $t_i^{\mf{B}}>0$ for every $i=1,\dots,m$.
\end{lemma}
\begin{proof}
We separate the proof in two steps. 


\paragraph{Step 1)} \emph{There exists $C_1>0$, depending on $V_i,\beta_{ii}$ and $ \max_{(i,j) \in \mathcal{K}_1} |\beta_{ij}|_\infty$, such that if $\mf{B}$ satisfies \eqref{ass_small_coop2} and $\mf{t}^{\mf{B}}$ is a maximum point of $\Psi_{\mf{B}}$, then $|\mf{t}^{\mf{B}}| \le C_1$.} \\
By contradiction, assume that this is not true. Then there exist $(\mf{B}^l)_l$, with $\mf{B}^l$ satisfying the \eqref{ass_small_coop2} for every $l$, and $\max_{(i,j) \in \mathcal{K}_1} |\beta_{ij}^l|_\infty \le C$ independently on $l$, such that:
\begin{itemize}
\item for every $l$ there exists $\mf{u}^l \in \mathbb{H}$ which is the limit (weak in $\mathbb{H}$, strong in $(L^4(\Omega))^d$, a.e. in $\Omega$) of a minimizing sequence $(\mf{u}^l_n)_n$ for $c_{\mf{B}^l}$; 
\item for every $l$ there exists $\mf{t}^l \in \overline{(\R_+)^m}$, maximum point of $\Psi_l:=\Psi_{\mf{B}^l}$, such that $|\mf{t}^l| \to +\infty$ as $l \to \infty$.
\end{itemize}
Having chosen $\mf{B}^l$ such that the \eqref{ass_small_coop2} holds,  for every $l$ we have $\mf{u}^l  \in \mathcal{E}_{\mf{B}^l}$, and in particular $\mf{M}^l:= \mf{M}(\mf{B}^l,\mf{u}^l)$ is positive definite. Let $\hat{\mf{t}}^l := \mf{t}^l/ |\mf{t}^l|$; we claim that 
\begin{equation}\label{limite versori}
\lim_{l \to \infty} \mf{M}^l \hat{\mf{t}}^l \cdot \hat{\mf{t}}^l = 0.
\end{equation}
Indeed, if this is not true, then there exists $\alpha>0$ such that up to a subsequence $\mf{M}^l \hat{\mf{t}}^l \cdot \hat{\mf{t}}^l\ge \alpha$, and by $|\mf{t}^l| \to +\infty$ we have
\[
\Psi_l(\mf{t}^l) = \frac{1}{2}\sum_{h} \|\mf{u}_{h}^l\|_h^2  t_{h}^l- \frac{1}{4} \mf{M}^l \mf{t}^l \cdot \mf{t}^l  \le C \overline{\gamma} |\mf{t}^l| - \frac{\alpha}{4} |\mf{t}^l|^2 \to -\infty
\]
as $l \to \infty$, in contradiction with the fact that $\Psi_l(\mf{t}^l) \ge \Psi_l(\mf{0}) = 0$ for every $l$. This proves the \eqref{limite versori}. At this point, we observe that the maximality of $\mf{t}^l$ entails
\begin{equation}\label{eq7pag12}
\Psi_l(\mf{t}^l) = \sup_{\theta>0} \Psi_l(\theta \hat{\mf{t}}^l) 
= \sup_{\theta>0} \left[ \frac{1}{2} \sum_{h=1}^m \|\mf{u}_{h}^l\|_h^2\, \hat t_{h}^l \theta - \frac{1}{4} \mf{M}^l \hat{\mf{t}}^l \cdot \hat{\mf{t}}^l \theta^2 \right].
\end{equation}
As function of $\theta$, the term in the brackets is of type $a_l \theta - b_l \theta^2$, with $a_l, b_l >0$, and $b_l \to 0$ by the \eqref{limite versori}. Moreover, thanks to Lemma \ref{lem: lower bound forte} (recall that $\max_{(i,j) \in \mathcal{K}_1} |\beta_{ij}^l|_\infty$ can be bounded independently on $l$) and to the fact that $|\hat{\mf{t}}^l|=1$ and $\hat t_{h}^l \ge 0$ for every $h$, it results
\begin{equation}\label{lower bound versori}
\sum_{h=1}^m \|\mf{u}_{h}^l\|_h^2\,  \hat{t}_{h}^l \ge S \underline{\gamma} \sum_{h=1}^m \hat{t}_{h}^l \ge C S\underline{\gamma}>0.
\end{equation}
Now, it is immediate to check that $\sup_{\theta>0} \left[a_l \theta - b_l \theta^2\right] =a_l^2/(4 b_l)$, so that collecting together \eqref{limite versori}, \eqref{eq7pag12}, and \eqref{lower bound versori}, we deduce
\begin{equation}\label{limit1}
\Psi_l(\mf{t}^l) = \frac{\left(\sum_{h=1}^m \|\mf{u}_{h}^l\|_h^2 \,\hat{t}_{h}^l\right)^2  }{4 \mf{M}^l \hat{\mf{t}}^l \cdot \hat{\mf{t}}^l} \to +\infty \qquad \text{as $l \to \infty$}.
\end{equation}

To reach a contradiction, we wish to show that, on the contrary, $\Psi_l(\mf{t}^l)$ is bounded in $l$. At first we observe that, for every $i \in I_h$, $h=1,\dots,m$, and $l$, it results $t_{h}^l u_{i,n}^l \wc t_{h}^l u_{i}^l$ weakly in $H_0^1(\Omega)$ as $ n \to \infty$. Thus the weak lower semi-continuity of $J_{\mf{B}^l}$ implies that 
\begin{align*}
\Psi_l(\mf{t}^l) & = J_{\mf{B}^l}\left( \sqrt{t_{1}^l}\mf{u}_{1}^l,\dots, \sqrt{t_{m}^l}\mf{u}_{m}^l \right) \\
& \le \liminf_{n \to \infty} J_{\mf{B}^l}\left( \sqrt{t_{1}^l}\mf{u}_{1,n}^l,\dots, \sqrt{t_{m}^l}\mf{u}_{m,n}^l \right) = \liminf_{n \to \infty} \Psi_{\mf{B}^l, \mf{u}_n^l}(\mf{t}^l).
\end{align*}
Since $\mf{u}_n^l \in \mathcal{N}_{\mf{B}^l} \cap \mathcal{E}_{\mf{B}^l}$, $\mf{1}$ is the unique maximum point of $\Psi_{\mf{B}^l,\mf{u}_n^l}$ in $\overline{(\R_+)^m}$, and thanks to Lemma \ref{lem: upper bounds} we can bound the right hand side independently on $n$ and $l$:
\[
\Psi_l(\mf{t}^l)  \le \liminf_{n \to \infty}  \Psi_{\mf{B}^l, \mf{u}_n^l}(\mf{t}^l) \le \liminf_{n \to \infty} \Psi_{\mf{B}^l, \mf{u}_n^l}(\mf{1}) = \liminf_{n \to \infty}  J_{\mf{B}^l}(\mf{u}_n^l) \le \overline{\gamma},
\]
in contradiction with the \eqref{limit1}.

\paragraph{Step 2)} \emph{Conclusion of the proof.} \\
The function $\Psi_{\mf{B}}$ is of class $\mathcal{C}^1 \left( \overline{(\R_+)^m}\right)$, see the \eqref{exp Psi}. So, if $\mf{t}^{\mf{B}}$ is a maximum point of $\Psi_{\mf{B}}$, we have $\pa_h \Psi_{\mf{B}}(\mf{t}^{\mf{B}}) = 0$ if $t_h^{\mf{B}}>0$, and $\pa_h \Psi_{\mf{B}}(\mf{t}^{\mf{B}})  \le  0$ if $t_h^{\mf{B}}=0$; that is,
\begin{align}
\|\mf{u}_h^{\mf{B}}\|_h^2 = \sum_{k=1}^m \mf{M}(\mf{B},\mf{u}^{\mf{B}})_{hk} t_k^{\mf{B}} \qquad &\text{if $t_h^{\mf{B}}>0$} \label{derivata in R_+} \\
\|\mf{u}_h^{\mf{B}}\|_h^2 \le \sum_{k\neq h} \mf{M}(\mf{B},\mf{u}^{\mf{B}})_{hk} t_k^{\mf{B}}\qquad &\text{if $t_h^{\mf{B}}=0$} \label{derivata su bordo}. 
\end{align}

Note that if $\max_{(i,j) \in \mathcal{K}_2 }   |\beta_{ij}^+|_\infty    =0$, namely if $\beta_{ij} \le 0$ for every $(i,j) \in \mathcal{K}_2$, the estimate \eqref{derivata su bordo} immediately gives a contradiction with Lemma \ref{lem: lower bound forte}; so, from now on we can suppose that $\max_{(i,j) \in \mathcal{K}_2 }   |\beta_{ij}^+|_\infty $ is strictly positive.

Let us assume by contradiction that there exists $\bar h =1,\dots,m$ such that $\mf{t}^{\mf{B}}$ is a maximum point of $\Psi_{\mf{B}}$ in $\overline{(\R_+)^m}$ and $t^{\mf{B}}_{\bar h}=0$. By $\pa_{\bar h} \Psi_{\mf{B}} (\mf{t}^\mf{B}) \le 0$, we deduce 
\begin{align*}
S \sum_{i \in I_{\bar h}} |u_i^{\mf{B}}|_4^2 & \le \|\mf{u}_{\bar h}^{\mf{B}}\|_{\bar h}^2 \le \sum_{k\neq {\bar h}} \mf{M}(\mf{B},\mf{u}^{\mf{B}})_{\bar hk} t_k^{\mf{B}}\\
& = \left( \max_{(i,j) \in \mathcal{K}_2 }   |\beta_{ij}^+|_\infty\right)  \left(\sum_{i \in I_{\bar h} } |u_i^{\mf{B}}|_4^2\right) \sum_{k \neq \bar h} \left(\sum_{j \in I_k } |u_j^{\mf{B}}|_4^2\right) t_k^{\mf{B}} ,
\end{align*}
which gives
\[
\left( \max_{(i,j) \in \mathcal{K}_2 }   |\beta_{ij}^+|_\infty\right) \sum_{k \neq \bar h} \left(\sum_{j \in I_k } |u_j^{\mf{B}}|_4^2\right) t_k^{\mf{B}} \ge S.
\]
In particular, there exists $\bar k \in \{1,\dots,m\} \setminus \{\bar h\}$, and $\bar j \in I_{\bar k}$, such that 
\begin{equation}\label{eq11pag14}
\left(\max_{(i,j) \in \mathcal{K}_2} |\beta_{ij}^+|_\infty \right) |u_{\bar j}^{\mf{B}}|_4^2 t_{\bar k}^{\mf{B}} \ge \frac{S}{d(m-1)},
\end{equation}
which implies $t^{\mf{B}}_{\bar k}>0$; hence $\pa_{\bar k} \Psi_{\mf{B}}(\mf{t}^{\mf{B}}) = 0$, i.e. by the \eqref{derivata in R_+}
\[
\mf{M}(\mf{B},\mf{u}^{\mf{B}})_{\bar k \bar k} t_{\bar k}^{\mf{B}} = \|\mf{u}_{\bar k}^{\mf{B}}\|_{\bar k}^2 - \sum_{h \neq \bar k} \mf{M}(\mf{B},\mf{u}^{\mf{B}})_{\bar k h} t_h^{\mf{B}} 
\]
The right hand side can be estimated using the \eqref{inequ Nehari limite}:
\begin{multline}\label{eq10pag15}
\mf{M}(\mf{B},\mf{u}^{\mf{B}})_{\bar k \bar k} t_{\bar k}^{\mf{B}} \le \sum_{h =1}^m \mf{M}(\mf{B},\mf{u}^{\mf{B}})_{\bar k h}- \sum_{h \neq \bar k}  \mf{M}(\mf{B},\mf{u}^{\mf{B}})_{\bar k h} t_h^{\mf{B}}\\
 \le  \sum_{h =1}^m  \sum_{(i,j) \in I_{\bar k} \times I_h} \int_{\Omega} \beta_{ij}^+ \left(u_i^{\mf{B}} u_j^{\mf{B}}\right)^2 
 + \sum_{h \neq \bar k}  \left( \sum_{(i,j) \in I_{\bar k} \times I_h} \int_{\Omega} \beta_{ij}^-\left(u_i^{\mf{B}} u_j^{\mf{B}}\right)^2 \right) t_{h}^{\mf{B}}\\
 \le \left( \max_{(i,j) \in \{1,\dots,d\}^2} |\beta_{ij}^+|_\infty \right) \frac{ \overline{\gamma}^2}{S^2} + C_1 \bar{M},
\end{multline}
where, in addition to usual arguments based on the H\"older inequality, we used Lemma \ref{lem: upper bounds} to bound the first term, and Lemma \ref{lem: bound competition} and the first step to bound the second term. Since $ \beta_{\bar j \bar j} \ge \mu_{\bar j} \ge \inf_j \mu_j =: C_2$ a.e. in $\Omega$, and $C_2>0$ depends only on $\beta_{ii}$, it results
\begin{multline*}
\frac{S}{ \displaystyle d(m-1) \max_{(i,j) \in \mathcal{K}_2} |\beta_{ij}^+|_\infty } \le |u_{\bar j}^{\mf{B}}|_4^2 t_{\bar k}^{\mf{B}} = \frac{|u_{\bar j}^{\mf{B}}|_4^2 \int_{\Omega} \beta_{\bar j \bar j}(x) \left(u_{\bar j}^\mf{B} \right)^4}{\int_{\Omega} \beta_{\bar j \bar j}(x) \left(u_{\bar j}^\mf{B} \right)^4} t_{\bar k}^{\mf{B}} \\
\le \frac{1}{\mu_{\bar j} |u_{\bar j}^{\mf{B}}|_4^2} \int_{\Omega} \beta_{\bar j \bar j}(x) \left(u_{\bar j}^\mf{B} \right)^4 t_{\bar k}^{\mf{B}} 
 \le \frac{d(m-1)}{S C_2  } \left(\max_{(i,j) \in \mathcal{K}_2} |\beta_{ij}^+|_\infty \right) \int_{\Omega} \beta_{\bar j \bar j}(x) \left(u_{\bar j}^\mf{B} \right)^4 \left(t_{\bar k}^{\mf{B}}\right)^2,
 \end{multline*} 
where in the firts and in the last inequality we used the \eqref{eq11pag14}. To continue the chain of inequalities, at first we note that the last integral can be estimated by $ \mf{M}(\mf{B},\mf{u}^{\mf{B}})_{\bar k \bar k}$; so, in a second time, we can apply the \eqref{eq10pag15} and the first step to obtain
%
\begin{multline*}
\frac{S}{ \displaystyle d(m-1) \max_{(i,j) \in \mathcal{K}_2} |\beta_{ij}^+|_\infty }  \le
\frac{d(m-1) t_{\bar k}^{\mf{B}}}{S C_2  } \left(\max_{(i,j) \in \mathcal{K}_2} |\beta_{ij}^+|_\infty \right) \mf{M}(\mf{B},\mf{u}^{\mf{B}})_{\bar k \bar k} t_{\bar k}^{\mf{B}} \\
 \le \frac{d(m-1)C_1}{S C_2  } \left(\max_{(i,j) \in \mathcal{K}_2} |\beta_{ij}^+|_\infty \right) \left[\left( \max_{(i,j) \in \{1,\dots,d\}^2} |\beta_{ij}^+|_\infty \right) \frac{ \overline{\gamma}^2}{S^2} + C_1 \bar{M}\right].
\end{multline*}
which gives a contradiction provided $\max_{(i,j) \in \mathcal{K}_2} |\beta_{ij}^+|_\infty$ is smaller than a constant depending only on $V_i,\beta_{ii}$ and $\max_{(i,j) \in \mathcal{K}_1} |\beta_{ij}|_\infty$.
\end{proof} 

The previous result implies that $(\sqrt{t_1^{\mf{B}}}\mf{u}_1^{\mf{B}} , \dots, \sqrt{t_m^{\mf{B}}} \mf{u}_m^{\mf{B}} ) \in \mathcal{N}_{\mf{B}}$ for any $\mf{t}^{\mf{B}}$ which is a maximum point of $\Psi_{\mf{B}}$. Thus, if we show that $\mf{1}$ is the unique maximum point of $\Psi_{\mf{B}}$, we obtain $\mf{u}^{\mf{B}} \in \mathcal{N}_{\mf{B}}$. This will be the object of the forthcoming Lemma \ref{lem: limite in nehari}; in the proof of such a result, we shall compare the value $J_{\mf{B}}(  \sqrt{t_1^{\mf{B}}}\mf{u}_1^{\mf{B}} , \dots, \sqrt{t_m^{\mf{B}}} \mf{u}_m^{\mf{B}})$ with $c_{\mf{B}}$, the infimum of $J_{\mf{B}}$ on $\mathcal{N}_{\mf{B}} \cap \mathcal{E}_{\mf{B}}$. We point out that at this stage such a comparison is not allowed, because we do not know that $(\sqrt{t_1^{\mf{B}}}\mf{u}_1^{\mf{B}} , \dots, \sqrt{t_m^{\mf{B}}} \mf{u}_m^{\mf{B}})\in \mathcal{E}_{\mf{B}}$. 

\begin{lemma}\label{lem: max in cone}
There exists $0<\bar K \le K_3$, depending only on $V_i,\beta_{ii}$ and \\
$\max_{(i,j) \in \mathcal{K}_1} |\beta_{ij}|_\infty$, such that if $\max_{(i,j) \in \mathcal{K}_2}  |\beta_{ij}^+|_\infty\le \bar K$, then any maximum point $\mf{t}^{\mf{B}}$ of $\Psi_{\mf{B}}$ is such that
$\left(\sqrt{t_1^{\mf{B}}}\mf{u}_1^{\mf{B}} , \dots, \sqrt{t_m^{\mf{B}}} \mf{u}_m^{\mf{B}} \right)\in \mathcal{E}_{\mf{B}}$.
\end{lemma}
\begin{proof}
We assume from the beginning that $\max_{(i,j) \in \mathcal{K}_2}  |\beta_{ij}^+|_\infty$ is smaller than $K_3$ defined in Lemma \ref{lem: max interno}. Under this assumption, $t_h^{\mf{B}}>0$ for every $h$.

\paragraph{Step 1)} \emph{Assume that 
\begin{equation}\label{ass_small_coop5}
\max_{(i,j) \in \mathcal{K}_2}  |\beta_{ij}^+|_\infty < \min\left\{K_3, S^3 \underline{\gamma}/(2\overline{\gamma}^2 C_1)\right\},
\end{equation}
where $S$, $\underline{\gamma}$, $\overline{\gamma}$ and $C_1$ have been defined by \eqref{def of S}, in Lemma \ref{lem: lower bound forte}, in Lemma \ref{lem: upper bounds} and in Step 1) of Lemma \ref{lem: max interno}, respectively. There exists $C_3>0$, depending only on $V_i$, $\beta_{ii}$ and $\max_{\mathcal{K}_1} |\beta_{ij}|_\infty$, such that $t_h^{\mf{B}} \ge C_3$ for every $h=1,\dots, m$.}\\
Assume by contradiction that the claim is not true. Then, there exist $(\mf{B}^l)_l$ with $\mf{B}^l$ satisfying the \eqref{ass_small_coop5} for every $l$, and $\max_{(i,j) \in \mathcal{K}_1} |\beta_{ij}^l|_\infty \le C$ independently on $l$, such that:
\begin{itemize}
\item for every $l$ there exists $\mf{u}^l \in \mathbb{H}$, which is the limit of a minimizing sequence for $c_{\mf{B}^l}$;
\item for every $l$ there exists $\mf{t}^l \in (\R_+)^m$, maximum point of $\Psi_l:=\Psi_{\mf{B}^l}$, such that $t_{\bar h}^l \to 0$ as $l \to \infty$ for some $\bar h \in \{1,\dots,m\}$.
\end{itemize}
As $\mf{B}^l$ satisfies the \eqref{ass_small_coop5}, by Lemma \ref{lem: max interno} we know that $t_{\bar h}^l>0$ for every $l$; by the \eqref{derivata in R_+}
\begin{align*}
S\underline{\gamma} &\le \sum_{i \in I_{\bar h}} \|u_{i}^l\|_i^2 = \sum_{k=1}^m \left(\sum_{(i,j) \in I_{\bar h} \times I_k} \int_{\Omega} \beta_{ij}^l(x) \left( u_{i}^l u_{j}^l\right)^2 \right) t_{k}^l \\
& \le \sum_{k=1}^m \left(\sum_{(i,j) \in I_{\bar h} \times I_k} |(\beta_{ij}^l)^+|_\infty |u_{i}^l|_4^2 |u_{j}^l|_4^2 \right) t_{k}^l \\
& \le \frac{\overline{\gamma}^2}{S^2}\left(\max_{(i,j) \in \mathcal{K}_1}  |\beta_{ij}^l|_\infty + \max_{i}  |\beta_{ii}|_\infty \right) t_{\bar h}^l + \frac{\overline{\gamma}^2}{S^2}\left( \max_{(i,j) \in \mathcal{K}_2}  |(\beta_{ij}^l)^+|_\infty\right) C_1,
\end{align*}
where the first inequality follows by Lemma \ref{lem: lower bound forte}, and the last one is a consequence of Lemma \ref{lem: upper bounds}, and of Step 1) of Lemma \ref{lem: max interno}. Now, passing to the limit as $l \to \infty$ in the previous chain of inequalities, since $t_{\bar h}^l \to 0$, the functions $\beta_{ii}$ are prescribed, and $\max_{\mathcal{K}_1} |\beta_{ij}^l|_\infty \le C$, we have
\[
S\underline{\gamma} \le \left(\max_{(i,j) \in \mathcal{K}_2}  |(\beta_{ij}^l)^+|_\infty\right) \frac{\overline{\gamma}^2}{S^2} C_1,
\]
which gives a contradiction with the \eqref{ass_small_coop5}.

\paragraph{Step 2)} \emph{There exists $\bar K< \min\{K_3, S^3 \underline{\gamma}/(2\overline{\gamma}^2 C_1)\}$ such that if $
\max_{(i,j) \in \mathcal{K}_2}  |\beta_{ij}^+|_\infty \le \bar K$, then $(\sqrt{t_1^{\mf{B}}} \mf{u}_1^{\mf{B}}, \dots, \sqrt{t_m^{\mf{B}}}  \mf{u}_m^{\mf{B}})  \in \mathcal{E}_{\mf{B}}$.} \\
Under the \eqref{ass_small_coop5}, by the first step not only $t_h^{\mf{B}}>0$, but also $t_h^{\mf{B}} \ge C_3 >0$ for every $h$. Hence, by the \eqref{derivata in R_+} and Lemma \ref{lem: lower bound forte}
\[
C_4:= C_3 S \underline{\gamma} \le t_h^{\mf{B}} \|\mf{u}_h^{\mf{B}}\|_h^2 = \sum_{k=1}^m \mf{M}(\mf{B},\mf{u}^{\mf{B}})_{hk} t_h^{\mf{B}} t_k^{\mf{B}},
\]
for every $h=1,\dots,m$. As a consequence
\[
  \mf{M}(\mf{B},\mf{u}^{\mf{B}})_{hh} \left(t_h^{\mf{B}} \right)^2 \ge C_4 - \sum_{k \neq h}  \mf{M}(\mf{B},\mf{u}^{\mf{B}})_{hk}  t_h^{\mf{B}} t_k^{\mf{B}} 
\]
for every $h=1,\dots,m$. We infer
\begin{align*}
\mf{M}(\mf{B},\mf{u}^{\mf{B}})_{hh} \left(t_h^{\mf{B}}\right)^2 &-  \sum_{k \neq h} |\mf{M}(\mf{B},\mf{u}^{\mf{B}})_{hk}  t_h^{\mf{B}} t_k^{\mf{B}} | \\
&\ge \mf{M}(\mf{B},\mf{u}^{\mf{B}})_{hh} \left(t_h^{\mf{B}}\right)^2  
- \sum_{k \neq h} \left(\sum_{(i,j) \in I_h \times I_k} \int_{\Omega} \left(\beta_{ij}^++ \beta_{ij}^- \right) \left(u_i^{\mf{B}} u_j^{\mf{B}} \right)^2\right) t_h^{\mf{B}} t_k^{\mf{B}}  \\
& \ge C_4 - 2 \sum_{k \neq h} \left(\sum_{(i,j) \in I_h \times I_k} \int_{\Omega} \beta_{ij}^+ \left(u_i^{\mf{B}} u_j^{\mf{B}} \right)^2\right) t_h^{\mf{B}} t_k^{\mf{B}}  \\
&\ge C_4 - 2  \left(\frac{C_1\overline{\gamma}}{S}\right)^2    \left( \max_{(i,j) \in \mathcal{K}_2}  |\beta_{ij}^+|_\infty \right),
\end{align*}
for every $h=1,\dots,m$, and the last term is positive provided $\max_{(i,j) \in \mathcal{K}_2}  |\beta_{ij}^+|_\infty$ is sufficiently small. 
\end{proof}

In the next lemma we show that, under the previous assumptions, $\mf{u}^{\mf{B}} \in \mathcal{N}_{\mf{B}} \cap \mathcal{E}_{\mf{B}}$.

\begin{lemma}\label{lem: limite in nehari}
If $\max_{(i,j) \in \mathcal{K}_2}  |\beta_{ij}^+|_\infty \le \bar K$
defined in Lemma \ref{lem: max in cone}, then $\mf{t}^{\mf{B}}=\mf{1}$ is the unique maximum point of $\Psi_{\mf{B}}$, and in particular $\mf{u}^{\mf{B}} \in \mathcal{N}_{\mf{B}} \cap \mathcal{E}_{\mf{B}}$.
\end{lemma}
\begin{proof}
In light of Lemma \ref{lem: max interno}, we know that \eqref{derivata in R_+} holds for every $h$. Comparing this with the \eqref{inequ Nehari limite}, we deduce that for every $h=1,\dots,m$.
\begin{equation}\label{eq15pag15}
\sum_{k=1}^m \mf{M}(\mf{B},\mf{u}^{\mf{B}})_{hk} \, t_k^{\mf{B}} \le \sum_{k=1}^m \mf{M}(\mf{B},\mf{u}^{\mf{B}})_{hk} 
\end{equation}
Now, since $\mf{u}^{\mf{B}}$ is the strong $(L^4(\Omega))^d$ limit of a minimizing sequence $(\mf{u}_n^{\mf{B}}) \subset \mathcal{N}_{\mf{B}} \cap \mathcal{E}_{\mf{B}}$, we have
\begin{equation}\label{eqpincopallino}
\frac{1}{4} \mf{M}(\mf{B},\mf{u}^{\mf{B}}) \mf{1}\cdot \mf{1}  = \lim_{ n \to \infty} \frac{1}{4} \mf{M}(\mf{B},\mf{u}_n^{\mf{B}}) \mf{1}\cdot \mf{1} = (\text{by \eqref{func on Nehari}}) = \lim_{n \to \infty} J_{\mf{B}}(\mf{u}^{\mf{B}}_n)= c_{\mf{B}}.
\end{equation}
By Lemmas \ref{lem: max interno} and \ref{lem: max in cone} we know that $(\sqrt{t_1^{\mf{B}}}\mf{u}_1^{\mf{B}},\dots, \sqrt{t_m^{\mf{B}}}\mf{u}_m^{\mf{B}})  \in \mathcal{N}_{\mf{B}} \cap \mathcal{E}_{\mf{B}}$. Thus
\begin{align*}
c_{\mf{B}} & \le J_{\mf{B}}\left(\sqrt{t_1^{\mf{B}}}\mf{u}_1^{\mf{B}},\dots, \sqrt{t_m^{\mf{B}}}\mf{u}_m^{\mf{B}}\right)  = \frac{1}{4} \mf{M}\left(\mf{B}, \left(\sqrt{t_1^{\mf{B}}}\mf{u}_1^{\mf{B}},\dots, \sqrt{t_m^{\mf{B}}}\mf{u}_m^{\mf{B}}\right)\right) \mf{1} \cdot \mf{1}\\
&= \frac{1}{4} \mf{M} (\mf{B},\mf{u}^{\mf{B}}) \mf{t}^{\mf{B}}\cdot \mf{t}^{\mf{B}} 
= \frac{1}{4}\sum_{h=1}^m \left[ \sum_{k=1}^m \mf{M} (\mf{B},\mf{u}^{\mf{B}})_{hk} \, t^{\mf{B}}_k \right] t_h^{\mf{B}} \le (\text{by \eqref{eq15pag15}}) \\
& \le \frac{1}{4}\sum_{h=1}^m \left[ \sum_{k=1}^m \mf{M} (\mf{B},\mf{u}^{\mf{B}})_{hk}\right] t_h^{\mf{B}}  = \frac{1}{4}\sum_{k=1}^m \left[ \sum_{h=1}^m \mf{M} (\mf{B},\mf{u}^{\mf{B}})_{hk} t_h^{\mf{B}} \right]  \\
&= \frac{1}{4}\sum_{k=1}^m \left[ \sum_{h=1}^m \mf{M} (\mf{B},\mf{u}^{\mf{B}})_{kh} t_h^{\mf{B}} \right] 
 \le (\text{by \eqref{eq15pag15}}) \le \frac{1}{4}\sum_{h,k=1}^m \mf{M} (\mf{B},\mf{u}^{\mf{B}})_{hk},
\end{align*} 
where we used again the \eqref{func on Nehari}, and also the symmetry of the matrix $\mf{M}(\mf{B},\mf{u}^{\mf{B}})$. A comparison with the \eqref{eqpincopallino} reveals that all the previous inequalities are, in fact, equalities, and in particular $=$ must hold in \eqref{eq15pag15}: $\mf{M}(\mf{B},\mf{u}^{\mf{B}}) \left( \mf{t}^{\mf{B}} -\mf{1} \right) = \mf{0}$.
As $\mf{u}^{\mf{B}} \in \mathcal{E}_{\mf{B}}$, the matrix $\mf{M} (\mf{B},\mf{u}^{\mf{B}})$ is non-singular, and the unique solution is $t_k^{\mf{B}}=1$ for every $k$.
\end{proof}

We are finally ready to complete the proof of Theorem \ref{thm:main1}.


\begin{proof}[Conclusion of the proof of Theorem \ref{thm:main1}]
By Lemmas \ref{lem: upper bounds}-\ref{lem: limite in nehari}, we know that there exists $\bar K>0$ such that if $\max_{(i,j) \in \mathcal{K}_2} |\beta_{ij}^+|_\infty \le \bar K$,
then there exists a minimizer $\mf{u}^{\mf{B}} \in \mathcal{N}_{\mf{B}} \cap \mathcal{E}_{\mf{B}}$ for $c_{\mf{B}}$, and it is possible to assume that $\mf{u}^{\mf{B}}$ has nonnegative components. Furthermore, for $\mf{u}^{\mf{B}}$ the estimates of points ($a$)-($c$) of Theorem \ref{thm:main1} are satisfied. Now, by Proposition \ref{prop: natural constraint}, $\mf{u}^{\mf{B}}$ is a free critical point of $J_{\mf{B}}$ in $\mathbb{H}$, and consequently is a solution of \eqref{sysx} with at least $m$ positive components. 
\end{proof}

We conclude this section with the proof of Proposition \ref{prop: constraint E omitted}.

\begin{proof}[Proof of Proposition \ref{prop: constraint E omitted}]
We show that under the assumption \eqref{relation pure competition}, the infimum $d_{\mf{B}}:= \inf_{\mathcal{N}_{\mf{B}}} J_{\mf{B}}$ is achieved, and that any minimizer belongs to $\mathcal{E}_{\mf{B}}$. Let $(\mf{u}_n^{\mf{B}})$ be a minimizing sequence for $d_{\mf{B}}$. As in Lemmas \ref{lem: upper bounds}, the sequence $(\mf{u}_n^{\mf{B}})$ is bounded in $\mathbb{H}$, and hence up to a subsequence it converges to some $\mf{u}^{\mf{B}}$ weakly in $\mathbb{H}$, strongly in $(L^4(\Omega))^d$, and a.e. in $\Omega$. Moreover, Lemma \ref{lem: lower bound forte} holds, \eqref{inequ Nehari limite} is satisfied and $J_{\mf{B}}(\mf{u}^{\mf{B}}) \le d_{\mf{B}}$. By the assumption \eqref{relation pure competition} and using Lemma \ref{lem: lower bound forte} and the inequality \eqref{inequ Nehari limite}, we immediately deduce that $\mf{u}^{\mf{B}} \in \mathcal{E}_{\mf{B}}$: indeed 
\[
0< S \underline{\gamma} \le \|\mf{u}_{h}^{\mf{B}}\|_h^2 \le \sum_{k=1}^m \mf{M}(\mf{B},\mf{u}^{\mf{B}})_{hk} = \mf{M}(\mf{B},\mf{u}^{\mf{B}})_{hh} - \sum_{k \neq h} |\mf{M}(\mf{B},\mf{u}^{\mf{B}})_{hk}|
\] 
for every $h$. As a consequence, there exists $\mf{t}^{\mf{B}} \in \overline{\mathbb{R}_+^m}$ which achieves the maximum of $\Psi_{\mf{B}}$, and, as observed at the beginning of Step 2) in Lemma \ref{lem: max interno}, thanks to \eqref{relation pure competition} it results that $t_h^{\mf{B}} >0$ for every $h$. At this point we can simply repeat step by step the proof of Lemma \ref{lem: limite in nehari}, with $c_{\mf{B}}$ replaced by $d_{\mf{B}}$ (this is possible because we know that $\mf{u}_{\mf{B}} \in \mathcal{E}_{\mf{B}}$, so that the matrix $\mf{M}(\mf{B},\mf{u}^{\mf{B}})$ is invertible), to deduce that $\mf{u}_{\mf{B}} \in \mathcal{N}_{\mf{B}}$, and hence it achieves $d_{\mf{B}}$. The fact that any minimizer for $d_{\mf{B}}$ belongs to $\mathcal{E}_{\mf{B}}$ is a trivial consequence of the definition of $\mathcal{N}_{\mf{B}}$ and of Lemma \ref{lem: lower bound forte}.
\end{proof}

\section{Phase separation for semi-trivial solutions}\label{sec: segregation}

In what follows, we consider the dependence of $\|\cdot\|_i$, $\|\cdot\|_h$, $J_{\mf{B}}$, $\mathcal{N}_{\mf{B}}$, \dots on $\mf{V}$. To emphasize it, we adopt the extended notation $\| \cdot\|_{V_i}$, $\|\cdot\|_{\mf{V}_h}$, $ J_{\mf{V},\mf{B}}$,$\mathcal{N}_{\mf{V},\mf{B}}$, \dots.

The setting we deal with is the following: let $\mf{a}$ be a $m$-decomposition of $d$, and let us consider sequences $(\mf{V}^n) \subset  (L^\infty(\Omega))^d$, $(\mf{B}^n) \subset (L^\infty(\Omega))^{d^2}$, with $\mf{B}^n$ symmetric for every $n$, such that assumptions \eqref{h0}, \eqref{h1} hold, and $\beta_{ij}^n \le 0$ a.e. in $\Omega$, for every $(i,j) \in \mathcal{K}_2$. Let $\mf{V}^\infty$ and $\mf{B}^\infty$ satisfying \eqref{h0} and \eqref{h1} as well (we use the notation $\mu_i^\infty$ to denote positive constant such that $\beta_{ii}^\infty \ge \mu_i^\infty$ a.e. in $\Omega$). We suppose that $\mf{V}^n \to \mf{V}^\infty$ in $(L^\infty(\Omega))^d$ and $\beta_{ij}^n \to \beta_{ij}^\infty$ in $L^\infty(\Omega)$ for every $(i,j) \in \mathcal{K}_1 \cup \{(i,i):i=1,\dots,d\}$, and the \eqref{ass segr} holds: $\beta_{ij}^n \to -\infty$ in $L^\infty(\Omega)$ for every $(i,j) \in \mathcal{K}_2$,  as $n \to \infty$. By Theorem \ref{thm:main1}, for every $n$ there exists $\bar K_n$ such that, if $\max_{(i,j) \in \mathcal{K}_2} |(\beta_{ij}^n)^+|_\infty < \bar K_n$, then there exists a minimizer for $J_{\mf{V}^n,\mf{B}^n}$ on $\mathcal{N}_{\mf{V}^n,\mf{B}^n} \cap \mathcal{E}_{\mf{B}^n}$. In light of the fact that $\beta_{ij}^n \le 0$ whenever $(i,j) \in \mathcal{K}_2$, this assumption is satisfied for every $n$, and we obtain a sequence $(\mf{u}^n)$ of minimizers which are semi-trivial solutions of \eqref{sysx} with potential $\mf{V}^n$ and coupling matrix $\mf{B}^n$. Moreover, the estimate \eqref{estimates thm 1} is satisfied. We let
\[
c_n:= J_{\mf{V}^n,\mf{B}^n}(\mf{u}^n) = \inf_{\mf{u} \in \mathcal{N}_{\mf{V}^n,\mf{B}^n} \cap \mathcal{E}_{\mf{B}^n}} J_{\mf{V}^n,\mf{B}^n}(\mf{u}).
\]

Now, let us recall the definitions of the matrix $\mf{M}_\infty(\mf{u})$, of the functional $J_{\infty}$, and of the Nehari-type manifold $\mathcal{N}_{\infty}$, see \eqref{def Minfty}-\eqref{Nehari limite}. In particular, we observe that $\mf{u} \in \mathcal{N}_\infty$ if $\|\mf{u}\|_{\mf{V}_h^\infty}>0$ and
\begin{equation}\label{caratterizzazione nehari limite}
\|\mf{u}_h\|_{\mf{V}_h^\infty}^2 = \mf{M}_\infty(\mf{u})_{hh}
\end{equation}
for every $h=1,\dots,m$. The functional on $\mathcal{N}_\infty$ reads
\begin{equation}\label{func on Nehari limite}
J_\infty(\mf{u}) = \frac{1}{4} \sum_{h=1}^m \mf{M}_\infty(\mf{u})_{hh}  = \frac{1}{4} \sum_{i=1}^d \|u_i\|_{V_i^\infty}^2 > 0 \qquad \forall \mf{u} \in \mathcal{N}_\infty.
\end{equation}

We set 
\[
c_\infty:= \inf \left\{ J_\infty(\mf{u}) \left| \begin{array}{l}
\mf{u} \in \mathcal{N}_\infty, \text{ and } \int_{\Omega} u_i^2 u_j^2 = 0 \\
\text{for every $(i,j) \in \mathcal{K}_2$}
\end{array} \right. \right\}.
\]

\begin{lemma}\label{lem: esiste minimo limite}
There exists $\bar{\mf{u}} \in \mathcal{N}_\infty$, such that $\int_{\Omega}\left( \bar u_{i} \bar u_{j}\right)^2 =0$ for every $(i,j) \in \mathcal{K}_2$, which achieves $c_\infty$.
\end{lemma}

\begin{proof}
By \eqref{func on Nehari limite} and the fact the constraint is not empty (it is possible to argue as in point 3) of Remark \ref{rem: geom constraint}), the minimization problem makes sense. Let $\tilde{\mf{u}} \in \mathcal{N}_\infty$ be such that $\int_{\Omega} \tilde{u}_{i}^2 \tilde{u}_{j}^2 =0$ for every $(i,j) \in \mathcal{K}_2$. Then $c_\infty \le J_\infty(\tilde{\mf{u}})$ and, as a consequence, any minimizing sequence $(\bar u_n)$ is bounded, and UTS it converges to some $\bar{\mf{u}}\in \mathbb{H}$ weakly in $\mathbb{H}$, strongly in $(L^4(\Omega))^d$, a.e. in $\Omega$, as $n \to \infty$. By the convergence, $\int_{\Omega}\left( \bar{u}_{i} \bar{u}_{j}\right)^2 =0$ for every $(i,j) \in \mathcal{K}_2$, and by \eqref{caratterizzazione nehari limite}
\begin{equation}\label{2101}
\|\bar{\mf{u}}\|_{\mf{V}^\infty_h}^2 \le \mf{M}_\infty(\bar{\mf{u}})_{hh}
\end{equation}
for every $h=1,\dots,m$. It is not difficult to modify the proof of Lemma \ref{lem: lower bound forte}, showing that there exist $C>0$ such that $\sum_{i \in I_h} |\bar{u}_i|_4^2 \ge C$ for every $h$. Hence, we can define $\bar{t}_h:=\|\bar{\mf{u}}_h\|_{\mf{V}_h^{\infty}}^2 / \mf{M}_\infty(\bar{\mf{u}})_{hh} >0$, so that by the \eqref{2101}
\[
\mf{M}_\infty(\bar{\mf{u}})_{hh} (\bar{t}_h-1) \le 0 \quad \Longrightarrow \quad \bar{t}_h \le 1 \qquad \text{for every $h=1,\dots,m$}.
\]
Moreover, by definition $( \sqrt{\bar{t}_1} \bar{\mf{u}}_1,\dots, \sqrt{ \bar{t}_m} \bar{\mf{u}}_m) \in \mathcal{N}_\infty$, and, clearly,
\[
\int_{\Omega} \left( \sqrt{\bar{t}_h} \bar{u}_{i} \sqrt{\bar{t}_h} \bar{u}_{j}\right)^2 =0 \quad \text{for every $(i,j) \in \mathcal{K}_2$}.
\]
Therefore, $4c_\infty \le  \sum_{h=1}^m \|\bar{\mf{u}}_h\|_{\mf{V}_h^\infty}^2 \bar{t}_h$; by weak lower semi-continuity, this implies
\[
\sum_{h=1}^m \|\bar{\mf{u}}_h\|_{\mf{V}_h^\infty}^2 \le \liminf_{n \to \infty} \sum_{h=1}^m \|\bar{\mf{u}}_{h,n}\|_{\mf{V}_h^\infty}^2 =\liminf_{n \to \infty} 4 J_{\infty}(\bar{\mf{u}}_n) = 4 c_\infty \le \sum_{h=1}^m \|\bar{\mf{u}}_h\|_{\mf{V}_h^\infty}^2 \bar{t}_h.
\]
As $\bar{t}_h \le 1$  for every $h$, necessarily $\bar{t}_h=1$ for every $h$, that is, $\bar{\mf{u}} \in \mathcal{N}_\infty$. In light of the weak lower semi-continuity of $J_\infty$, the thesis follows.
\end{proof}

The previous lemma can be used to relate the values of $c_n$ and $c_\infty$.

\begin{lemma}\label{rem: c_n e c_infty}
It results $\displaystyle \limsup_{n \to \infty} c_n \le c_\infty$.
\end{lemma}
\begin{proof}
Let $\bar{\mf{u}}$ be defined in the Lemma \ref{lem: esiste minimo limite}. For $h=1,\dots,m$ and $n \in \N$, we let $t_h^n:= \|\bar{\mf{u}}_h\|_{\mf{V}_h^n}^2/\mf{M}(\mf{B}^n,\bar{\mf{u}})_{hh}>0$. As $\bar u_i \bar u_j \equiv 0$ for every $(i,j) \in \mathcal{K}_2$, it results $( \sqrt{t_1^n} \bar{\mf{u}}_1,\dots, \sqrt{ t_m^n} \bar{\mf{u}}_m)$ $\in \mathcal{N}_{\mf{V}^n,\mf{B}^n} \cap \mathcal{E}_{\mf{B}^n}$ for every $n$. Since $\mf{V}^n \to \mf{V}^\infty$, $\beta_{ij}^n \to \beta_{ij}^\infty$ for every $(i,j) \in \mathcal{K}_1\cup \{(i,i)\}$ uniformly in $\Omega$, and being $\bar{\mf{u}} \in \mathcal{N}_\infty$, it results $t_h^n \to 1$ as $n \to \infty$, for every $h$. Therefore, 
\[
\limsup_{n \to \infty} c_n \le \limsup_{ n \to \infty} J_{\mf{V}^n,\mf{B}^n} \left( \sqrt{t_1^n} \bar{\mf{u}}_1,\dots, \sqrt{ t_m^n} \bar{\mf{u}}_m\right)  = J_\infty(\bar{\mf{u}})=c_\infty. \qedhere
\]
\end{proof}

Now, let us analyse the behaviour of the sequence $(\mf{u}^n)$, composed by minimizers for $c_n$. Since $(c_n) \subset \R_+$ is bounded, up to a subsequence there exists $\lim_n c_n= \lim_n \|\mf{u}_n\|_{\mf{V}^n}^2/4$. As observed in the \eqref{stima ainfty}-\eqref{stima cinfty}, the estimates \eqref{stime al limite} in Theorem \ref{thm:main1} can be made uniform in $n$: indeed
\begin{align}
&\sum_{i=1}^d \int_{\Omega} |\nabla u_i^n|^2 \le \sum_{i=1}^d \|u_i^n\|_{V_i^n}^2 \le 2\bar C \frac{\left( 1+ 2\max_i |V_i^\infty|_\infty\right)^2}{\min_i \mu_i^\infty} =: \overline{\gamma}^\infty   \label{stima ainfty}\\
&\sum_{i\in I_h} |u_i^n|_4^2 \ge \frac{S}{2d \left(\max_{(i,j) \in \mathcal{K}_1} |(\beta_{ij}^\infty)^+|_\infty + \max_{i} |\beta_{ii}^\infty|_\infty\right)}=: \underline{\gamma}_\infty \label{stima binfty}\\
& \sum_{i,j=1}^d \int_{\Omega} (\beta_{ij}^n)^- (u_i^n u_j^n)^2 \le  2  \left( \max_{(i,j) \in \mathcal{K}_1} |(\beta_{ij}^\infty)^+|_\infty  + \max_{i} |\beta_{ii}^\infty|_\infty\right) \left(\frac{\overline{\gamma}^\infty}{S}\right)^2 =: \bar{M}^\infty, \label{stima cinfty}
\end{align}

 By the \eqref{stima ainfty}-\eqref{stima binfty}, up to a subsequence $\mf{u}^n \wc \mf{u}^\infty$ in $\mathbb{H}$, $\mf{u}^n \to \mf{u}^\infty$ strongly in $(L^\infty(\Omega))^d$, $\mf{u}^n \to \mf{u}^\infty$ a.e. in $\Omega$, and $\mf{u}^\infty$ is such that $\|\mf{u}^\infty_h\|_{\mf{V}^\infty_h}^2>0$ for every $h$. Moreover, in light of assumption \eqref{ass segr} and the \eqref{stima cinfty}, for every $(i,j) \in \mathcal{K}_2$
\begin{equation}\label{limite segregato}
\int_{\Omega} \left(u^\infty_i u_j^\infty\right)^2 = \lim_{n \to \infty} \int_{\Omega} \left(u_i^n u_j^n\right)^2 = 0 \quad \Longrightarrow \quad u_i^\infty u_j^\infty \equiv 0 \quad \text{a.e. in $\Omega$}. 
\end{equation}
We wish to show that $\mf{u}^\infty \in \mathcal{N}_\infty$, and is a minimizer for $c_\infty$.

%

\begin{lemma}
The function $\mf{u}^\infty$ belongs to $\mathcal{N}_\infty$.
\end{lemma}
\begin{proof}
We introduce the auxiliary function $\Psi_\infty:\overline{(\R_+)^m} \to \R$
\begin{align*}
\Psi_\infty(\mf{t})& := J_\infty\left( \sqrt{t_1} \mf{u}_1^\infty, \dots, \sqrt{t_m} \mf{u}_m^\infty \right)  = \frac{1}{2} \sum_{h=1}^m \|\mf{u}^\infty\|_{\mf{V}^\infty_h}^2 t_h - \frac{1}{4}  \sum_{h=1}^m  \mf{M}_\infty(\mf{u}^\infty)_{hh} t_h^2.
\end{align*}
It is then clear that $\mf{t}^\infty>0$ defined by
\begin{equation}\label{eqA.4}
\|\mf{u}^\infty\|_{\mf{V}^\infty_h}^2 - \mf{M}_\infty(\mf{u}^\infty)_{hh} t_h^\infty  = 0
\end{equation}
is the unique critical point of $\Psi_\infty$ in $(\R_+)^m$, and that, being $\mf{M}_\infty(\mf{u}^\infty)$ positive definite, $\mf{t}^\infty$ is a strict maximum. Clearly $( \sqrt{t_1^\infty} \mf{u}_1^\infty, \dots, \sqrt{t_m^\infty} \mf{u}_m^\infty) \in \mathcal{N}_\infty$, so we aim at proving that $t_h^\infty=1$ for every $h$. Since $\mf{u}^n \in \mathcal{N}_{\mf{V}^n,\mf{B}^n}$, for $h=1,\dots,m$ we have
\[
\|\mf{u}_h^n\|_{\mf{V}^n_h}^2 = \sum_{k=1}^m \mf{M}(\mf{B}^n,\mf{u}^n)_{hk} \quad \Longrightarrow \quad \|\mf{u}_h^\infty\|_h^2  - \mf{M}_\infty(\mf{u}^\infty)_{hh} \le \limsup_{n \to \infty} \sum_{k \neq h} \mf{M}(\mf{B}^n,\mf{u}^n)_{hk} \le 0,
\]
where the last inequality follows by the fact that $\beta_{ij}^n \le 0$ a.e. in $\Omega$ for every $(i,j) \in \mathcal{K}_2$. A comparison with the \eqref{eqA.4} reveals that $t_h^\infty \le 1$ for every $h=1,\dots,m$. Now, as $( \sqrt{t_1^\infty} \mf{u}_1^\infty, \dots, \sqrt{t_m^\infty} \mf{u}_m^\infty) \in \mathcal{N}_\infty$ and by the \eqref{func on Nehari limite} and \eqref{limite segregato}, we have $4 c_\infty \le \sum_h \|\mf{u}_h^\infty\|_{\mf{V}^\infty_h}^2 t_h^\infty$. Therefore, the variational characterization of $\mf{u}^n$, the \eqref{func on Nehari}, and Lemma \ref{rem: c_n e c_infty} give
\[
\sum_{h=1}^m \|\mf{u}_h^\infty\|_{\mf{V}^\infty_h}^2 \le \lim_{n \to \infty} \sum_{h=1}^m \|\mf{u}_h^n\|_{\mf{V}^n_h}^2 = \lim_{n \to \infty} 4 c_n \le 4 c_\infty   \le \sum_{h=1}^m \|\mf{u}_h^\infty\|_{\mf{V}^\infty_h}^2 t_h^\infty.
\]
As $t_h^\infty \le 1$ for every $h$, we deduce that necessarily $t_h^\infty=1$ for every $h$, that is, $\mf{u}^\infty \in \mathcal{N}_\infty$. 
\end{proof}

\begin{proof}[Conclusion of the proof of Theorem \ref{thm:main3}]
We start showing that the convergence of $\mf{u}^n$ to $\mf{u}^\infty$ is strong in $\mathbb{H}$. As $\mf{u}^\infty \in \mathcal{N}_\infty$ and $\mf{u}^n \in \mathcal{N}_{\mf{B}^n}$,
\begin{equation}\label{eqpagA6}
\begin{split}
\|\mf{u}_h^\infty\|_{\mf{V}^\infty_h}^2 &= \mf{M}_\infty(\mf{u}^\infty)_{hh} = \lim_{n \to \infty} \mf{M}(\mf{B}^n,\mf{u}^n)_{hh} \\
& \ge \lim_{n \to \infty}  \sum_{k=1}^m \mf{M}(\mf{B}^n,\mf{u}^n)_{hk} = \lim_{n \to \infty} \|\mf{u}_h^n\|_{\mf{V}^n_h}^2,
\end{split}
\end{equation}
where we used the fact that $\beta_{ij}^n \le 0$ a.e. in $\Omega$ for every $(i,j) \in \mathcal{K}_2$. On the other hand, by the convergence of $\mf{u}^n$ to $\mf{u}$ and of $\mf{V}^n$ to $\mf{V}^\infty$, also the opposite inequality holds, so that
\begin{equation}\label{eq2pagA6}
\|\mf{u}_h^\infty\|_{\mf{V}^\infty_h}^2= \lim_{n \to \infty} \|\mf{u}_h^n\|_{\mf{V}^n_h}^2\qquad h=1,\dots,m.
\end{equation}
This and the weak $(H_0^1(\Omega))^d$ convergence $\mf{u}^n \wc \mf{u}^\infty$ imply that $\mf{u}^n \to \mf{u}^\infty$ strongly in $\mathbb{H}$. Now, thanks to the \eqref{eq2pagA6}, the inequality \eqref{eqpagA6} is an equality, and in particular
\[
\lim_{n \to \infty}  \sum_{k \neq h} \mf{M}(\mf{B}^n,\mf{u}^n)_{hk}  =0 \quad h=1,\dots,m \quad \Longrightarrow \quad \lim_{n \to \infty} \sum_{(i,j) \in \mathcal{K}_2} \int_{\Omega} \left(\beta_{ij}^n\right)^- \left(u_i^n u_j^n\right)^2 = 0.
\]
This and the \eqref{eq2pagA6} permit to infer
\[
c_\infty \le J_\infty(\mf{u}^\infty) = \lim_{n \to \infty} J_{\mf{B}^n}(\mf{u}^n) = \lim_{n \to \infty} c_n,
\]
and by Lemma \ref{rem: c_n e c_infty} we conclude that $J_\infty(\mf{u}^\infty) = c_\infty = \lim_n c_n$. 
\end{proof}

\section{Positive solutions for systems with strong competition and strong cooperation}\label{sec: strong}

This section is devoted to the proofs of Theorems \ref{thm:main4} and \ref{thm:theorem 6}. In both the situations we study the constrained second differential of $J_{\mf{B}}$ on $\mathcal{N}_{\mf{B}}$.

\begin{remark}\label{rem: constrained differential}
Assume that $\tilde{\mf{u}} \in \mathcal{N}_{\mf{B}}$ is a free critical point of $J_{\mf{B}}$ in the whole $\mathbb{H}$, so that $d J_{\mf{B}}(\tilde{\mf{u}})=0$ in $\mathbb{H}^*$ (the dual space of $\mathbb{H}$), and let $\gamma: I \subset \R \to \mathcal{N}_{\mf{B}}$ be any smooth curve with support on the constraint $\mathcal{N}_{\mf{B}}$ and such that $\gamma(0) = \tilde{\mf{u}}$. It results
\[
\left.\frac{d^2 J_{\mf{B}}(\gamma(t))}{d t^2}\right|_{t=0} = \left.\left( d^2 J_{\mf{B}}(\gamma(t))[\gamma'(t),\gamma'(t)] + d J_{\mf{B}}(\gamma(t))\gamma''(t) \right)\right|_{t=0} = d^2 J_{\mf{B}}(\tilde{\mf{u}})[\gamma'(0),\gamma'(0)].
\]
Since $\mathcal{N}_{\mf{B}}$ is a $\mathcal{C}^2$ manifold, any tangent vector can be represented by a $\mathcal{C}^2$ curve, and the previous computation reveals that \emph{the constrained second differential of $J_{\mf{B}}$ on $\mathcal{N}_{\mf{B}}$, evaluated in a free critical point $\tilde{\mf{u}}$ of $J_{\mf{B}}$ which belongs to $\mathcal{N}_{\mf{B}}$, is equal, as quadratic operator on $T_{\tilde{\mf{u}}} \mathcal{N}_{\mf{B}}$, to the free second differential of $J_{\mf{B}}$ in $\tilde{\mf{u}}$}: 
\[
D^2_{\mathcal{N}_{\mf{B}}} J_{\mf{B}}(\tilde{\mf{u}}) \left[ \mf{v},\mf{v}\right]= D^2 J_{\mf{B}}(\tilde{\mf{u}}) \left[ \mf{v},\mf{v}\right] \qquad \forall \mf{v} \in T_{\tilde{\mf{u}}} \mathcal{N}_{\mf{B}}^2.
\]
\end{remark}

\begin{proof}[Proof of Theorem \ref{thm:main4}]
Assume by contradiction that the statement is not true. Then there exist $(\mf{V}^n) \subset (L^\infty(\Omega))^d$ and $(\mf{B}^n) \subset  (L^\infty(\Omega))^{d^2}$ such that $V_i^n \to \tilde V_h$ and $\beta_{ij}^n \to \tilde \beta_h$ in $L^\infty(\Omega)$ for every $i \in I_h$ and $(i,j) \in I_h^2$ respectively, for every $h$; $\beta_{ii}^n= \beta_{ii}$ for every $i$,
%
%
$\beta_{ij}^n \to -\infty$ in $L^\infty(\Omega)$ for every $(i,j) \in \mathcal{K}_2$, and the \eqref{th variazionale} is not satisfied by $\mf{V}^n$, $\mf{B}^n$, that is,
\begin{multline}\label{absurd assumtpion}
c_n:= \inf\left\{J_{\mf{V}^n,\mf{B}^n}(\mf{u}): \mf{u} \in \mathcal{N}_{\mf{V}^n,\mf{B}^n} \cap \mathcal{E}_{\mf{B}^n} \right\} \\ \ge
\inf\left\{J_{\mf{V}^n,\mf{B}^n}(\mf{u}) \left| \begin{array}{l}
\mf{u} \in \mathcal{N}_{\mf{V}^n,\mf{B}^n} \cap \mathcal{E}_{\mf{B}^n} 
\text{ and there exists} \\
 \text{$i =1,\dots,d$ such that $u_i \equiv 0$}
\end{array} \right.\right\}.
\end{multline}
Now, note that in the present setting by Theorem \ref{thm:main1} there exists $\mf{u}^n \in \mathcal{N}_{\mf{V}^n,\mf{B}^n} \cap \mathcal{E}_{\mf{B}^n}$ which achieves $c_n$, and it is a free critical point of $J_{\mf{V}^n,\mf{B}^n}$ in $\mathbb{H}$. In light of the \eqref{absurd assumtpion}, we can assume that for every $n$ there exists $i_n$ such that $u_{i_n}^n \equiv 0$ in $\Omega$. Since $\{1,\dots,m\}$ is discrete and finite, it is not restrictive to assume that $i_n = \bar i$ for every $n$. As $\mf{u}^n$ is a free critical point of $J_{\mf{V}^n,\mf{B}^n}$, recalling what we observed in Remark \ref{rem: constrained differential} we have
\begin{equation}\label{diff secondo magg 0}
\begin{split}
0 &\le D^2_{\mathcal{N}_{\mf{B}^n}} J_{\mf{V}^n,\mf{B}^n} (\mf{u}^n)[\mf{v}^n,\mf{v}^n] = D^2 J_{\mf{V}^n,\mf{B}^n} (\mf{u}^n) [\mf{v}^n,\mf{v}^n] \\
& = \sum_{i =1}^d \|v_i^n\|_{V_i^n}^2- \sum_{i,j=1}^d \int_{\Omega} \beta_{ij}^n \left(u_i^n v_j^n\right)^2 - 2\sum_{i,j=1}^d \int_{\Omega} \beta_{ij}^n u_i^n u_j^n v_i^n v_j^n,
\end{split}
\end{equation}
for every $\mf{v}^n \in T_{\mf{u}^n} \mathcal{N}_{\mf{B}^n}$ and for every $n$. Let $\bar h$ be such that $\bar i \in I_{\bar h}$. By \eqref{stima binfty}, there exist $\underline{\gamma}$ and $\bar j \in I_{\bar h}$ such that $|u_{\bar j}^n|_4^2 \ge \underline{\gamma}/d$ for every $n$. By \eqref{diff constraint}, it is easy to check that if $\mf{v}^n$ is defined by
\[
v^n_i:= \begin{cases}
0 & \text{if $i \neq \bar i$} \\
u^n_{\bar j} & \text{if $i = \bar i$},
\end{cases}
\] 
it results $\langle \nabla G_{\mf{B}^n,h}(\mf{u}^n),\mf{v}^n \rangle = 0$ for every $h$ and $n$, and as a consequence $\mf{v}^n \in T_{\mf{u}^n} \mathcal{N}_{\mf{V}^n,\mf{B}^n}$ for every $n$. Thus by \eqref{diff secondo magg 0} we have
\[
0 \le  \|u_{\bar j}^n\|_{V_{\bar i}^n}^2 - \sum_{k \neq \bar i} \int_{\Omega} \beta_{\bar i k}^n (u_{\bar j}^n  u_k^n )^2,
\]
for every $n$. Passing to the limit as $n \to \infty$, since $\beta_{ij}^n \to -\infty$ in $L^\infty(\Omega)$ for every $(i,j) \in \mathcal{K}_2$ we can apply Theorem \ref{thm:main3}: it implies that $\beta_{\bar i k}^n (u_k^n u_{\bar j}^n )^2 \to 0$ in $L^1(\Omega)$ as $n \to \infty$ for every $k \not \in I_{\bar h}$, so that
\begin{equation}
\begin{split}\label{assurdo al limite}
0 &\le \lim_{n \to \infty} \left( \|u_{\bar j}^n\|_{V_{\bar i}^n}^2 - \sum_{k \in I_{\bar h} \setminus \{\bar i\}} \int_{\Omega} \beta_{\bar i k}^n (u_{\bar j}^n  u_k^n )^2 - \sum_{k \in \{1,\dots,d\} \setminus I_{\bar h}} \int_{\Omega} \beta_{\bar i k}^n (u_{\bar j}^n  u_k^n )^2 \right) \\
& = \|u_{\bar j}^\infty\|_{\tilde V_{\bar h}}^2 - \sum_{k \in I_{\bar h} \setminus \{\bar i\}} \int_{\Omega} \tilde \beta_{\bar h} (u_{\bar j}^\infty  u_k^\infty )^2.
\end{split}
\end{equation}
where $\mf{u}^\infty \in \mathbb{H}$ is a minimizer for the corresponding limit minimization problem. Now, as $\mf{u}^n$ is a solution of system \eqref{sysx} with potentials $\mf{V}^n$ and coupling matrix $\mf{B}^n$, we can test the equation for $u_{\bar j}^n$ with $u_{\bar j}^n$ itself, and passing to the limit we deduce
\begin{equation}\label{limite equazioe uj}
\|u_{\bar j}^n\|_{V_{\bar i}^n}^2  = \sum_{k=1}^d \int_{\Omega} \beta_{\bar j k}^n (u_k^n u_{\bar j}^n )^2 \quad \Longrightarrow \quad \|u_{\bar j}^\infty\|_{\tilde V_{\bar h}}^2 = \sum_{k \in I_{\bar h} \setminus \{\bar i, \bar j\}} \int_{\Omega} \tilde \beta_{\bar h} (u_{\bar j}^\infty  u_k^\infty )^2 + \int_{\Omega} \beta_{\bar j \bar j} (u_{\bar j}^\infty)^4,
\end{equation}
where we used again Theorem \ref{thm:main3}. Plugging the \eqref{limite equazioe uj} into the \eqref{assurdo al limite}, we obtain a contradiction:
\[
0 \le \int_{\Omega} ( \beta_{\bar j \bar j} - \tilde \beta_h) (u_{\bar j}^\infty)^4 <0,
\]
where in the last step we used assumption \eqref{h3} and the fact that $|u_{\bar j}^\infty|_4^2 \ge \underline{\gamma}$.
\end{proof}

\begin{proof}[Proof of Theorem \ref{thm:theorem 6}]
Up to minor changes, the same notation and the same line of reasoning adopted in the previous proof yields
\begin{equation}\label{assurdo al limite2}
\begin{split}
0 &\le \lim_{n \to \infty} \left( \|u_{\bar j}^n\|_{V_{\bar i}}^2 - \sum_{k \in I_{\bar h} \setminus \{\bar i\}} \int_{\Omega} \beta_{\bar i k} \left(u_{\bar j}^n  u_k^n \right)^2 - \sum_{k \in \{1,\dots,d\} \setminus I_{\bar h}} \int_{\Omega} \beta_{\bar i k}^n \left(u_{\bar j}^n  u_k^n \right)^2 \right) \\
& = \|u_{\bar j}^\infty\|_{\tilde V_{\bar i}}^2 -  \int_{\Omega} \beta_{\bar i \bar j} \left(u_{\bar j}^\infty\right)^4,
\end{split}
\end{equation}
where $\mf{u}^\infty$ is given by Theorem \ref{thm:main3}, and we used the fact that $I_{\bar h} \setminus \{\bar i\}$ has a unique element. Since we are assuming $u_{\bar i}^n \equiv 0$ for every $n$, by the equation for $u_{\bar j}^n$, and in light of Theorem \ref{thm:main3}, it is not difficult to deduce that $\|u_{\bar j}^\infty\|_{V_{\bar j}}^2 = \int_{\Omega} \beta_{\bar j\bar j} (u_{\bar i}^\infty)^4$. Now, recalling that $C_{\bar i \bar j}$ denotes the best constant such that $\|u\|_{V_{\bar i}}^2 \le C \|u\|_{V_{\bar j}}^2 $ for every $u \in H_0^1(\Omega)$, we can estimate the last term in the \eqref{assurdo al limite2} in the following way:
\[
0 \le  C_{\bar i \bar j} \|u_{\bar j}^\infty\|_{V_{\bar j}}^2 - \int_{\Omega} \beta_{\bar i\bar j} (u_{\bar j}^\infty)^4 \le \int_{\Omega} ( C_{\bar i \bar j} \beta_{\bar j \bar j} - \beta_{\bar i\bar j} ) (u_{\bar j}^\infty)^4 <0,
\]
where the last inequality follows by assumption \eqref{ass strong coop 2}.
\end{proof}

\section{Further results and comments}\label{sec: further}

\paragraph{Extension of our results in higher dimension.} All our results can be extended with some efforts to problems of type
\[
\begin{cases}
- \Delta u_i + V_i(x) u_i = \sum_{j = 1}^d \beta_{ij}(x) |u_j|^{q} |u_i|^{q-2} u_i & \text{in $\Omega$} \\
u_1=\cdots=u_d=0 &\text{in $\Omega$}.
\end{cases}
\]
where $\Omega$ is either a bounded domain of $\R^N$, or $\Omega=\R^N$, with $N \ge 2$; it is sufficient to assume that $1<q<N/(N-2)$ if $N \ge 3$, or $q>1$ if $N = 1,2$. We prefer to consider \eqref{sysx} in light of the great attention which has been devoted in the literature to systems of that form.

%

\paragraph{Regularity for minimizers of the limit problem in Theorem \ref{thm:main3}}

Let $\mf{u}^\infty$ be the limiting profile of Theorem \ref{thm:main3}. By the occurrence of phase-separation, the domain $\Omega$ segregates in $m$ components $\Omega_h:= \bigcup_{i \in I_h}\left\{u_i^{\infty}>0\right\}$. We think that it can be interesting to analyse the regularity of $\mf{u}^\infty$ and of the free-boundary. Similar analysis has been carried on in \cite{CoTeVe2003} in a completely competitive framework, and in \cite{CafLin} assuming that $\beta_{ii}= 0$ for every $i$ and $\beta_{ij} \le 0$ for every $i \neq j$.

\paragraph{Acknowledgements.} We thank Professor Susanna Terracini for having suggested the problem, and for several inspiring discussions. Moreover,  we are indebted with Hugo Tavares for some precious suggestions.


\end{document}